\documentclass[11pt]{amsart}

\usepackage{amsmath}
\usepackage{amssymb}
\usepackage{graphicx}
\usepackage{amsthm}
\usepackage{subfig, amsmath, hyperref}
\usepackage{mathrsfs}
\usepackage[usenames, dvipsnames]{xcolor}
\def\bel{\begin{equation}\label}
\def\eeq{\end{equation}}
\def\bel{\begin{equation}\label}
\def\eeq{\end{equation}}

\newtheorem{theorem}{Theorem}[section]

\newtheorem{lemma}[theorem]{Lemma}
\newtheorem{proposition}[theorem]{Proposition}
\theoremstyle{definition}
\newtheorem{definition}[theorem]{Definition}
\newtheorem{example}[theorem]{Example}
\newtheorem{remark}[theorem]{Remark}
\newtheorem{Claim}{Claim}[section]

\numberwithin{equation}{section}

\def \M {{M_0}}

\def \R {{\mathbb R}}
\def \k {{k}}

\def \d {{\bf d}}

\def \vv{\vskip 0.5 truecm}

\def \dbar {{\bar d}}
\def \I {{\mathcal I}}

\def \RR {{\mathbb R}}
\def \NN {{\mathbb N}}
\def \P {{\F}}
\def \F {{\mathcal F}}

\def \p {{p_0}}
\def \E {{\F}^{max}}

\def \Ubalr {{U_r}}

\def \Wbalr {{\bar U_r}}

\def\bel{\begin{equation}\label}
\def\eeq{\end{equation}}

\def\T{\mathbf C}

 \def \M {{M_0}}

\def \R {{\mathbb R}}
\def \N {{\mathbb N}}
\def \k {{k}}

\def \F {{f}}

\def \p {{p_0}}
\def \E {{f}^{max}}

\def\bel{\begin{equation}\label}
\def\eeq{\end{equation}}

\def\T{\mathbf C}

\def \R {{I\!\!R}}
\def \RR {{I\!\!R}}

\title[MRF for unbounded dynamics;
  control-polynomial systems]{
Minimum Restraint Functions \\
for   unbounded dynamics:
\\
general and control-polynomial systems
}

\author[A. C. Lai]{Anna Chiara Lai}
\address{A. C. Lai, Dipartimento di Scienze di Base e Applicate per l'Ingegneria,
Sapienza Universit\`a di Roma\\ Via Scarpa 16, Roma 00181, Italy}
\email{{anna.lai@sbai.uniroma1.it}}

\author[M. Motta]{Monica Motta}
\address{M. Motta, Dipartimento di Matematica,
Universit\`a di Padova\\ Via Trieste, 63, Padova  35121, Italy}
 \email{{motta@math.unipd.it}}

\author[F. Rampazzo]{Franco Rampazzo}
\address{F. Rampazzo, Dipartimento di Matematica,
Universit\`a di Padova\\ Via Trieste, 63, Padova  35121, Italy}
\email{{rampazzo@math.unipd.it}}

\thanks{ This research is partially supported by  the Gruppo Nazionale per l' Analisi Matematica, la Probabilit\`a e le loro Applicazioni (GNAMPA) of the Istituto Nazionale di Alta Matematica (INdAM), Italy;  and 
by Padova University grant PRAT 2015 ``Control
of dynamics with reactive constraints''}


\begin{document}

\begin{abstract}
We consider   an exit-time minimum  problem with a running cost \, $l\geq 0$ and unbounded controls.  The occurrence  of points where $l=0$  can be regarded as a transversality loss. Furthermore, since controls  range over  unbounded sets,  the family of admissible  trajectories  may lack important  compactness   properties. In the first part of the paper we show that the existence of  a {\it  $\p$-Minimum Restraint Function} provides not only global asymptotic controllability (despite non-transversality)  but also  a state-dependent upper bound for the value function (provided $p_0>0$). This extends to  unbounded  dynamics a former  result which heavily relied  on the compactness   of the control set. 

 In the second part of the paper we apply the general result  to  the case when the system  is  {\it polynomial} in the control variable.  Some elementary, algebraic,  properties of the convex hull of   vector-valued polynomials' ranges allow some simplifications of  the  main result, in terms of either  near-affine-control systems  or  reduction to  weak subsystems  for the original dynamics.
\end{abstract}

 \maketitle

\section{Introduction} 
Mainly motivated  by the case when the dynamics  is polynomial in the control,
we deal with  optimal control problems of the form 
\begin{align}
&\dot x=f(x,u), \qquad x(0) = z  \label{Eintro}, \\&(x(t),u(t)) \in (\Omega\backslash{\bf C})\times U , \quad \lim_{t\to {T}_{x}^-}  \d (x(t),\T)=0,\\
&   \displaystyle{\I(x,u):= \int_ 0^{T_{x}} l(x(t),u(t))\, dt,} \qquad  \displaystyle{V(z):=  \inf_{(x,u)} \I(x,u)},\label{minprobintro}
\end{align}
where: i) for given positive integers $n,m$,  the {\it state space}  $\Omega$ is an open subset of $ \RR^n $, the {\it  controls} $u$  range over a (possibly unbounded)  subset of $U\subseteq\R^m$,   and  ${\bf C}\subset\Omega$  is a    closed  {\it target} with compact boundary; 
ii) the {\it current cost}  $l(x,u)$ is $\geq 0$ for all $ (x,u) \in (\Omega\backslash{\bf C})\times U$;
iii) $T_{{x}} \in [0,+\infty]$ is  the infimum of  times needed for the trajectory  $x(\cdot)$  to approach the target $\T$; and
iv) $\mathbf d(x,{\bf C})$ denotes the usual (Euclidean) distance of the point $x$ from the subeset $\bf C$. 

We focus on a particular kind of Lyapunov function, called {\it $p_0$-Minimum Restraint Function } ($\p\geq 0$). This notion  has  been introduced in \cite{MR13}  under the extra-hypothesis that the controls range over a  bounded  set. The existence  of a $p_0$-Minimum Restraint Function, besides implying global asymptotic controllability  to $\T$, was shown to  provide  a continuous  upper estimate for the  value function  $V$. Such an estimate is not trivial, in that the problem (here and in  \cite{MR13} as well)  lacks what in first order PDE's is called   {\it  transversality},  which  would correspond to the assumption  $l(x,u) \neq 0 $ for all  $(x,u)$ (as in the  minimal time problem, where $l=1$)\footnote{But here the exit time can well be infinite.}.
Here, we  extend  the concept  of $p_0$-Minimum Restraint Function to  unbounded   dynamics $f$. Notice that the unboundedness of $f$ (and $l$) cannnot be neglected, for no  {\it  coercivity}  hypotheses --roughly speaking,  the fact that $u\mapsto l(x,u)$ grows suitably faster than $u\mapsto f(x,u) $ --  rule out the need  of larger  and larger velocities in a minimizing sequence.

Precisely,  for a $\p\geq 0$   we call \emph{$\p$-Minimum Restraint Function}  every         continuous       function          $$W:\Omega\setminus\overset{\circ}\T\to[0,+\infty[$$ whose restriction to $\Omega\setminus\T$ (is   locally semiconcave, positive definite and proper \footnote{See  Definition  \ref{defMRF}, where,  as soon as $\Omega\subsetneq \R^n$, one also posits
 $W_0\in \R\cup \{+\infty\}$ such that
 $W (\Omega\setminus\T)< W_0$ \text{and} $\lim_{x\to x_0} W(x)=W_0$,
 for every     $x_0\in\partial\Omega$.}, and)  verifies
  \bel{MRH1intro}
 H_{l,f}(x,\p,D^*W(x))<0 \quad \forall x\in \Omega\setminus\T,
 \eeq
where the Hamiltonian $H_{l,\F}$ is defined by 
 \bel{Hamintro }
 H_{l,f}(x,p_0,p):= \inf_{u\in U}\Big\{ \langle p ,f(x,u) \rangle+p_0\,l(x,u)\Big\}  .
 \eeq

The inequality  \eqref{MRH1intro} has to be interpreted as  $H_{l,\F}(x,\p, p )<0$  $\forall p\in D^*W(x)$---which includes  the case $H_{l,\F}(x,\p, p )=-\infty$ . 
The following hypothesis will be crucial: 
\vskip0.3truecm
 {\bf Hypothesis A}:  {\it  For every compact subset ${\mathcal K}\subset\Omega\backslash\T$
the function
\bel{ipointro}
(\bar l,\bar{\F})(x,u) := \frac{(l,{\F})}{1 +|(l,\F)(x,u)|}  (x,u)\eeq
 is uniformly continuous on ${\mathcal K}\times U$.}

\vskip0.3truecm
Observe  that Hypothesis {\bf A} allows for  a vast class of cost-dynamic pairs $(l,f)(x,u)$\footnote{See  Remark \ref{A'}, for   a bit stronger hypothesis.}, including ($x$-dependent) polynomials in  $u_1,\cdots,u_m$,
$|u_1|,\cdots,|u_m|$, $|u|$, and compositions of polynomials with exponential and Lipschitz continuous functions.
Let us bring forward the  statement of our main result:
\begin{theorem}\label{stimaenergiaintro}
  Assume Hypothesis {\bf A} and let $W$ be   a $\p$-Minimum Restraint Function for the problem $(l,f,\T)$, for some $\p\ge0$. Then
 \begin{itemize}
 \item[{\bf (i)}]   system {\rm\eqref{Eintro}} is   globally asymptotically controllable to $\T$.
 \end{itemize} 
Furthermore, \begin{itemize}\item[{\bf (ii)}] if     $\p>0$,  then
  \bel{}
 {{V}}(z)\le \frac{W(z)}{{{\p}}}\,  \qquad\forall z\in \Omega\setminus \T.
\eeq
 \end{itemize} 
  \end{theorem}
The proof of the theorem relies on a state-based time rescaling of the problem,  which in turn is made possible by Hypothesis  {\bf A}.  The controls of the  rescaled problem (see Section \ref{generalsec}) still range in the (possibly unbounded) set $U$. Yet, some compactness properties 
 of the  rescaled dynamics are of  crucial importance  in the construction of  trajectories reaching the target at least asymptotically.

An application to the gyroscope (see Subsection \ref{gyroex})  concludes  Section \ref{generalsec}: an explicit  $\p$-Minimum Restraint Function is provided for a minimum  problem where  the {control is identified with the  pair made by the  precession and spin velocities,  while the state  corresponds to pair made by the  nutation angle and its time-derivative. 
\vskip0.5truecm
The remaining part of the paper is devoted to problems whose dynamics can be parameterized  by a $u$-polynomial:
\bel{polintro}
\dot x = \displaystyle f (x,u):= f_0(x)+\sum_{i=1}^d\left(\sum_{\alpha\in\N^m, \, \alpha_1+\dots+\alpha_m=i} u_1^{\alpha_1}\cdots u_m^{\alpha_m} f_{\alpha_1,\dots,\alpha_m}(x)\right)\,.
\eeq

Among   applications for which the polynomial dependence is relevant let us mention Lagrangian mechanical systems, possibly with friction forces, in which   inputs  are identified with the derivatives of  some Lagrangian  coordinates. In this case $d=2$\footnote{This is clearly a consequence of the fact that the  kinetic energy is a  quadratic form of the velocity (see, besides Subsection \ref{gyroex}, \cite{aldobressan} and \cite{BR10}).}.  We point out also that, in connection with the investigation of uniqueness and regularity of solutions for Hamilton-Jacobi equations, dynamics and current costs with  unbounded controls and { polynomial growth}  have been already addressed in \cite{mot04},  \cite{MS14}, by embedding the problem in a space-time problem  through  techniques of graph's reparameterization -- see e.g.
\cite{BR88,BR10,Dyk90,Gur72,Mil96,Ris65,RS00,VP88}.  With similar arguments (see also \cite{MR03}) necessary conditions for the existence of (possibly impulsive) minima of input-polynomial optimal control problems  have been studied in \cite{chez}.  Furthermore, the interplay between convexity and polynomial dependence of both the dynamics and the running cost \, has been investigated also  in \cite{PT09}, in connection with  problems  of existence of optimal solutions.

A careful  investigation  of elementary,  algebraic properties of the convex hull $co~f(x,\R^m)$  proves essential for the application of  Theorem
\ref{stimaenergiaintro} to the polynomial case \eqref{polintro}. 
For instance,  we consider {\it near-control-affine} control systems,  a  class of control-polynomial systems where the convex hull  of the dynamics can be parameterized as a  control-affine system with controls in a neighborhood of the origin\footnote{Once the convex hull of the dynamics is so nicely  parameterized,  relaxation   arguments allow applying several well-established results for control-affine systems.}. For instance,   this is  clearly false for the system
$
\dot x = f_0(x) +uf_1(x)+u^2 f_2(x),
$
$u\in\R,$
-- because the origin $(0,0)$ does not belong to the  the  convex hull's interior of the  curve  $(u,u^2)$. Instead, in view of Theorem \ref{cnear-control-affine},  the  convex hull of the image of 
\begin{multline*}
f_0(x) +u_1 f_{1,0,0,0,0,0,0}(x) + u_1u^5_3f_{1,0,5,0,0,0,0}(x) +u_2^3u_6^3 f_{0,3,0,0,0, 3,0}(x) \\+u_1u_3^5u_7^9 f_{1,0,5,0,0,0,9}(x),  
 \end{multline*}

$(u_1,\dots,u_7) \in\R^7$ does  coincide with the range of
$$
 f_0(x) +w_1f_{1,0,0,0,0,0,0}(x)  + w_2 f_{1,0,5,0,0,0,0}(x)  +w_3  f_{0,3,0,0,0, 3,0}(x)  +w_4 f_{1,0,5,0,0,0,9}(x),   
$$
$(w_1,w_2,w_3,w_4)\in\R^4.$

\vskip0.3truecm
When the system is not near-control-affine (and $U=\RR^m$), one  can try to exploit {\it weak subsystems}: the latter are    selections of the set-valued function $x\mapsto co~f(x,\R^m)$. In particular,  we  consider the  \emph{maximal degree subsystem}  and, for any $\lambda$ in the $m$-dimensional  simplex,  the  \emph{$\lambda$-diagonal subsystems}  (see Definition \ref{diagdef}
and Subsection \ref{secmaximal}, respectively). The  idea of utilizing {\it subsystems} might look counterproductive  with respect to the task of finding  a $\p$-Minimum Restraint Function: indeed, for such a purpose, having a sufficiently large amount  of  available directions plays crucial.   However, from a practical perspective, a diminished complexity in the dynamics might  ease the  guess of a $\p$-Minimum Restraint Function, which would automatically  be  a  $\p$-Minimum Restraint Function for the original polynomial problem.
To give  the flavour  of this viewpoint,  let us  anticipate a result (see Theorem \ref{maximalth} for details) concerning 
 maximal degree subsystems. 

 \begin{theorem}\label{thselectionintro} Let the growth assumption specified in   {Hypothesis} {\bf A$_{max}$}  below (Section \ref{secmaximal}) be verified.
If $W$ is a  $\p$-Minimum Restraint Function   for the {\rm maximal degree subsystem}
$$f^{max}(x,u):=f_0(x)+\sum_{\alpha\in\N^m, \, \alpha_1+\dots+\alpha_m=d} u_1^{\alpha_1}\cdots u_m^{\alpha_m} f_{\alpha_1\dots\alpha_m}(x),$$
 then $W$  is also a  $\p$-Minimum Restraint Function for the original control polynomial system
$$\displaystyle f (x,u):= f_0(x)+\sum_{i=1}^d\left(\sum_{\alpha\in\N^m, \, \alpha_1+\dots+\alpha_m=i} u_1^{\alpha_1}\cdots u_m^{\alpha_m} f_{\alpha_1,\dots,\alpha_m}(x)\right)\,.$$
 \end{theorem}

 \vskip0.3truecm
The paper is organized as follows.  In the remaining part of the present  section  we  provide  some preliminary definitions and notation. In Section \ref{generalsec} we prove  Theorem \ref{stimaenergiaintro} and exhibit  a $\p$-Minimum Restraint Function for  the gyroscope (see Subsection \ref{gyroex}).    Section \ref{rescsec}  is entirely devoted to  the proof of Theorem \ref{nuovo} which deals with a suitably rescaled problem.  In Section \ref{balsect} we focus on the case when the system  is   polynomial  in the control variable.   An Appendix with a technical proof concludes the paper.

 \vskip0.3truecm
\subsection{Preliminary concepts and notation}\label{subpb}   
$\,$
 \vskip0.3truecm
Let us   gather some notational conventions as well as  some basic concepts and results  which will be used throughout the paper. 
\vv
We are given an open set $\Omega\subset\RR^n$ and a {\it target }$\T\subset \Omega$, which we assume to have compact boundary $\partial\T$.   
For brevity, let us use the notation $\d(x)$ in place of $\d(x,\T)$.

\begin{definition}\label{xadm}
 We say that a path $x:[0,T_x[\to\Omega$ is  {\rm admissible} if
 \begin{itemize}
\item[i)] $0< T_x\leq+\infty$,
\item[ii)] $x\in AC_{loc}([0,T_x[,\Omega)$, 
\item[iii)] $x([0,T_x[)\subset \Omega\backslash \T$,
\item[iv)]  $\displaystyle \lim_{t\to  T^-_{ x}} {\bf d}(x(t))=0$. 
\end{itemize}
We call $T_x$ the {\rm exit time of $x$ from $\Omega\setminus \T$}. 
\end{definition}
Notice that the limit of $x(\cdot)$ for $t\to T_x^-$ need not exist, even when $T_x<+\infty$. Of course, if the limit exists, then it belongs to the target $\T$.

\begin{definition}\label{adm}
Let $g:\Omega\times U\to\R^n$ be a continuous function. For every  $z\in\Omega\setminus\T$, we will say that $(x,u)$ is an  {\rm   admissible trajectory-control pair} from $z$  for the control system 
 \bel{Egen}
 \dot x=g(x,u), \quad x(0)=z
 \eeq
  if 
\begin{itemize}
\item[i)] $x:[0,T_x[\to \Omega\setminus\T$ is an  admissible path,
\item[ii)] $u(\cdot)\in L^\infty_{loc}([0,T_x[,U)$,
\item[iii)] $x(\cdot)$ is a Charath\'eodory solution\footnote{Notice that such a solution  might be not unique.}  of {\rm\eqref{Egen}} corresponding to the input $u$.
\end{itemize}  
We shall use ${\mathcal{ A}}_g({ z})$ to denote the family of 
admissible trajectory-control pairs from $z$ for the control system {\rm\eqref{Egen}}.
\end{definition}

\vskip0,3truecm
  As customary, we shall use ${\mathcal KL}$ to denote the set of all continuous functions  $$\beta:[0,+\infty[\times[0,+\infty[\to[0,+\infty[$$ such that: (1)\, $\beta(0,t)=0$ and $\beta(\cdot,t)$ is strictly increasing and unbounded for each $t\ge0$; (2)\, $\beta(r,\cdot)$ is decreasing  for each $r\ge0$; (3)\, $\beta(r,t)\to0$ as $t\to+\infty$ for each $r\ge0$. 

 \begin{definition}\label{(GAC)}
The system   (\ref{Egen}) is   {\em globally asymptotically controllable to $\T$}  -- shortly, (\ref{Egen}) is {\em GAC to $\T$} --  provided  there is a  function $\beta\in{\mathcal KL}$  such that,  for each initial state $z\in\Omega\setminus \T$,    there exists an admissible  trajectory-control pair $(x,u)\in {\mathcal{ A}}_g({ z})$   that verifies
\bel{bbound}
{\bf d}(x(t))\le \beta({\bf d}(z),t) \qquad \forall t\in[0,+\infty[.\,\,
\footnote{ By convention, we fix an arbitrary $\bar z\in\partial\T$ and  formally establish that, if $T_{x}<+\infty$, the trajectory $x(\cdot)$ is prolonged to $[0,+\infty[$,
by setting $x(t)=\bar z$ for all $t\geq T_{x}$.}
\eeq
\end{definition}

\vv
\begin{definition}[Positive definite and proper functions] Let ${\bf E}$, $\Theta\subset \RR^n$ be,  respectively,   a closed and an open set with ${\bf E}\subset \Theta$  and let  $F:\Theta\setminus\overset{\circ}{\bf E}\to\R$  be a continuous function. Then $F$ is {\em positive definite on $\Theta\setminus {\bf E}$} if  $F(x)>0$ for all $ x\in\Theta\setminus {\bf E}$ and $F(x)=0$ for all $ x\in\partial{\bf E}$. \\
The function $F$ is called {\em proper  on $\Theta\setminus {\bf E}$} if the pre-image $F^{-1}(K)$ of any compact set $K\subset[0,+\infty)$ is compact.
\end{definition}

 \begin{definition}[Semiconcave functions]\label{sconc} Let $\Theta\subseteq\R^n$.  A continuous function  $F:\Theta\to\R$  is said to be {\rm  semiconcave  on $\Theta$} if  
 $$
 F(z_1)+F(z_2)-2F\left(\frac{z_1+z_2}{2}\right)\le \rho|z_1-z_2|^2, 
 $$
 for all $z_1$, $z_2\in \Theta$ such that $[z_1,z_2]\subseteq\Theta$.  $F$ is said to be {\rm  locally semiconcave  on $\Theta$} if it semiconcave on every compact subset of $\Theta$. 
 \end{definition}
 
We remind that locally semiconcave functions are locally Lipschitz continuous. 

 \begin{definition}[{\rm Limiting gradient}]\label{D*} Let $\Theta\subset\R^n$ be an open set and let   $F:\Theta\to\R$  be a locally Lipschitz function.  For every $x\in \Theta$ we set
$$
D^*{F}(x) := \Big\{ w\in\R^n\mid \ \  w=\lim_{k}\nabla {F}(x_k), \ \  x_k\in DIFF(F)\setminus\{x\}, \ \ \lim_k x_k=x\Big\}
$$
where   $\nabla$ denotes the classical gradient operator and $DIFF(F)$ is the set of differentiability points of $F$.  $D^*{F}(x)$ is called
the  \emph{set of limiting gradients} of $F$ at $x$.
\end{definition}

\begin{remark}
{\rm  
The set-valued map $x\mapsto D^*F(x)$ is upper semicontinuous on $\Theta$,  with non-empty, compact values. Notice that   $D^*{F}(x)$ is not convex.   
  When $F$ is a locally semiconcave function,  $D^*{F}$ coincides  with the limiting subdifferential $\partial_LF$, namely, 
  $$
  D^*F(x)=\partial_LF(x) := \{\lim \,  p_i: \  p_i\in \partial_PF(x_i), \ \lim\,  x_i=x\} \quad \forall x\in\Theta,
  $$   
where  $\partial_PF$ denotes the proximal subdifferential,  largely used in the literature on Lyapunov functions. }
\end{remark}

Basic properties of the semiconcave functions  imply  the following fact:
\begin{lemma}\label{Lscv} Let $\Theta\subset\R^n$ be an open set and let  $F:\Theta \to\R$   be a   locally  semiconcave function. Then for any compact set $\mathcal{K}\subset \Theta$ there exist some positive constants  $L$ and $\rho$ such that, for any $x\in \mathcal{K}$ 
 \footnote{The inequality (\ref{scvintro}) is usually formulated with the proximal superdifferential  $\partial^P F$. However, this does not make a difference here since $\partial^P F=\partial_C F=co D^* F$ as soon as $F$ is locally semiconcave. Hence (\ref{scvintro}) is true in particular for $D^*F$. }, 
\bel{scvintro}
\begin{array}{l}
F(\hat x)-F(x)\le  \langle p,\hat x-x \rangle+\rho|\hat x-x|^2, \\ \, \\
|p|\le L    \quad \forall p\in D^*F(x),
\end{array}
\eeq
for any point  $\hat x\in\mathcal{K}$ such that $[x,\hat x]\subset  \mathcal{K}$.   
\end{lemma}

\section{$\p$-Minimum restraint functions}\label{generalsec}
\subsection{The main result}
$\,$ 
\vskip0.3truecm
Let us begin with a precise formulation of the  minimum problem. For every initial condition $z\in\Omega\setminus\T$, we consider the control system 
 \bel{E}
 \dot x=f(x,u), \quad x(0)=z,
 \eeq 
 and,  for any admissible trajectory-control pair $(x,u)\in{\mathcal A}_f(z)$ (see Definition \ref{adm}),  let us  introduce the payoff
\begin{equation}\label{1.2}
\I(x,u):= \int_ 0^{T_x} l(x(t),u(t))\, dt \qquad (T\in]0,+\infty]).
\end{equation} 
The corresponding {\it  value function} is given by
\begin{equation}\label{minprob}
V(z)\:= 
 \inf_{(x,u)\in{\mathcal{ A}}_f({z})}\I( x, u ) \quad(\le+\infty).
\end{equation}
\vskip0.3truecm

Recall our principal  hypothesis:
\vskip0.3truecm

 {\bf Hypothesis A}:  {\it  For every compact subset ${\mathcal K}\subset\Omega\backslash\T$
the function
\bel{risc}
(\bar l,\bar{\F})(x,u) := \frac{(l,{\F})}{1 +|(l,\F)(x,u)|}  (x,u)
\eeq
 is uniformly continuous on ${\mathcal K}\times U$.}
\vskip0.3truecm

\begin{remark}\label{A'} {\rm As observed in the Introduction, this hypothesis  allows for  a wide set of {\it unbounded} dynamics and running costs. 
Furthermore,  it is easy to check that the following  condition  {\it  is sufficient for  Hypothesis {\bf A}} to hold true:

$\,$

\begin{itemize}
\item[]{\it The map ~$(l,\F)$ is continuous with respect to the state variable $x$ and locally Lipschitz with respect to the control variable $u$, and
$$
\left|\frac{D_u(l,\F)}{(1+|(l,\F)|)^2}\right|(x,u) \leq \eta(x) \qquad \text{ for a.e.  } (x,u) \in
(\Omega\backslash\T)\times U,
$$ for  some  continuous function
$\eta:~\Omega\backslash\T\to~[0,+\infty[ $.}
\end{itemize}}

\end{remark}

\vv  
  
 \vskip0.3truecm
Let us extend the definition  of  $\p$-Minimum Restraint Function (\cite{MR13}) to the case of unbounded control sets.

 \begin{definition}\label{defMRF}
   Let $W:\Omega\setminus\overset{\circ}{\T}\to[0,+\infty[$ be a continuous function, and let us assume that $W$ is  locally semiconcave, positive definite, and  proper on $\Omega\setminus\T$.
We say that $W$  is a \emph{$p_0$-Minimum Restraint Function --in short, $\p$-MRF--  for  $( l, \F,\T)$    in $\Omega$
 for some $\p\ge0$}  if
    \bel{MRH1} H_{l,\F}(x,\p, D^*W(x) )<0 \quad \forall x\in {{\Omega}\setminus\T}  \quad\footnote{This  means that $H_{l,\F}(x,\p, p )<0$ for every $p\in D^*W(x)$.}
    \eeq
 and, moreover, there exists  $W_0\in [0,+\infty]$, such that
 $$W(\Omega\setminus {\T})< W_0\quad \text{and} \quad \lim_{x\to x_0,\ x\in \Omega} W(x)=W_0$$
 for every $x_0\in\partial\Omega$. 
\end{definition}
 
 \vskip0.3truecm
We can now state our main result:
  \vskip0.2truecm
\noindent{\bf Theorem 1.1.}\label{MRFth} {\it Assume  Hypothesis {\bf A}  and  let $W$ be   a $\p$-Minimum Restraint Function for the problem $(l,f,\T)$,  for some $\p\ge0$. Then:
 \begin{itemize}
 \item[{\bf (i)}]   system {\rm \eqref{E}} is   globally asymptotically controllable to $\T$;
  \item[{\bf (ii)}] if     $\p>0$,  then
  \bel{Wprop}
 {{V}}(z)\le \frac{W(z)}{{{\p}}}\,  \qquad\forall z\in \Omega\setminus \T.
\eeq 
\end{itemize}}
 \vv
 
 \begin{proof}
We begin with a state-based rescaling procedure. Precisely, we consider the optimal control  problem 
\bel{peoblemrep}
\begin{array}{l}
\displaystyle y'(s) =
\bar \F(y,v) \quad y(0) = z;\\\,\\
\displaystyle \I_{\bar l,\bar\F}(y,v) :=  \int_0^{S_y}\bar l(y(s),v(s)) ds, \qquad \bar V(z) :=\displaystyle\inf_{(y,v)\in {\mathcal{ A}}_{\bar f}({z})} \I_{\bar l,\bar\F}(y,v),  
\end{array}\eeq
where   $\bar l$, $\bar{\F}$ are  defined in \eqref{risc}, 
the apex  denotes differentiation with respect to the parameter $s$,  and $S_y\le+\infty$  is the exit time  of the admissible trajectory $y(\cdot)$  (in the time parameter $s$). 

The connection between the original optimal control problem and the rescaled one is established by the following result.

 \vskip0.2truecm
 \begin{Claim}\label{cres}  The path  $(y,v)$ is an   admissible trajectory-control pair for {\rm \eqref{peoblemrep}} if and only if, setting 
\begin{align*}
&t(s):=\int_0^s (1+|(l,\F)(y(\eta),v(\eta)|)^{-1} d\eta \qquad\forall s\in [0,S_{y }[\\
&x(t):= y \circ s(t)\qquad u(t):= v\circ s(t)\qquad\ \forall t\in[0,T_x[, \quad T_x:=t(S_{y }),
 \end{align*}
the path  $(x,u)$ is an admissible  trajectory-control pair for {\rm \eqref{E}}--{\rm \eqref{minprob}}. Furthermore,  
$$
\displaystyle\int_0^{S_{y }}  \bar l(y(s),v(s)) ds = \displaystyle\int_0^{T_x}   l(x(t),u(t)) dt.
$$
In particular, one has  
$$V(z)=\bar V(z)$$ 
for all $z\in \Omega\setminus\T$.
\end{Claim}
 \vskip0.2truecm
 Indeed, since $t=t(s)$ is absolutely continuous and $t'(s)>0$  almost everywhere,   the inverse map $s(\cdot)= t^{-1}(\cdot)$ is  absolutely continuous  (see e.g. \cite[Theorem 4, page 253]{Nat57} or, for a more general statement, \cite[Theorem 2.10.13, page 177]{Fed69}).  In particular, $x=y\circ s$ is absolutely continuous, and $u= v\circ s$ turns out to be Borel measurable as well.
Hence the claim  follows by a  standard application  of the chain rule\footnote{  Notice that  the solutions to $\dot x=f$ or $\dot y =\bar\F$ are not necessarily unique.}.

 \vskip0.2truecm  
The Hamiltonian $H_{\bar l,\bar \F}$ associated to $\bar l$, $\bar \F$,
$$
H_{\bar l,\bar \F}(x,p_0,p):= \inf_{u\in U}\Big\{ \langle p ,\bar \F(x,u) \rangle+p_0\,\bar l(x,u)\Big\} 
$$
for all $(x,p_0,p)\in (\Omega\backslash\T)\times \R^{1+n},$ is continuous and  sublinear in $(p_0,p)$, uniformly with respect to $x$.
Furthermore, it is also trivial to check  that, for every $(x,p_0,p)\in(\Omega\setminus\T)\times\RR^{1+n}$,
\bel{HH}
H_{\bar l,\bar{\F}}(x,p_0,p)<0 \quad\iff\quad
 H_{l,\F }(x,p_0,p)< 0.
\eeq
In particular, for every $\p\ge0$   $W$ is a $\p$-MRF for $(l,f,\T)$    if and only if  $W$ is a $\p$-MRF for $(\bar l,\bar{f},\T)$. Moreover,  because of Hypothesis {\bf A}, the problem $(\bar l,\bar\F,\T)$ meets   the hypotheses of Theorem \ref{nuovo} below. Therefore:

 \vskip0.2truecm  {\it  \begin{itemize}
 \item[(i)] if there exists  a $\p$-MRF $W$   for $(l,f,\T)$, then the  rescaled system in {\rm \eqref{peoblemrep}} is   GAC to $\T$, i.e. there exists a function $\beta\in{\mathcal K\mathcal L}$  such that for any $z\in\Omega\setminus\T$ there is an admissible trajectory-control pair  $(y,v)\in{\mathcal A}_{\bar f}(z)$  that verifies
\bel{bbound1}
 {\bf d}(y(s))\le \beta({\bf d}(z),s) \qquad \forall s\in[0,+\infty[;
\eeq
  \item[(ii)] moreover, if   $\p>0$,  then
  \bel{Wprop1}
 {{\bar V}}(z)\le \frac{W(z)}{{{\p}}}.
\eeq
\end{itemize}}
 \vskip0.2truecm
 
If $x(\cdot)$ is the trajectory defined in Claim \ref{cres},   one then  obtains
\bel{bbound2}
{\bf d}(x(t))\le\beta({\bf d}(z),s(t)) \qquad \forall t\in[0,+\infty[
\eeq
and, if $\p>0$,
  \bel{Wprop2}
 {{V}}(z)\le \frac{W(z)}{{{\p}}}\,  \qquad\forall z\in \Omega\setminus \T.
\eeq
Notice that $t(s)\leq s$ for all $s$, so that $t\leq s(t)$ for all $t$.
Since  the map $\beta(z,\cdot)$ is decreasing, one gets
 $$ \beta(z,s(t))\leq \beta(z,t)$$
for all $t$. It follows by \eqref{bbound1} that
\bel{bbound3} {\bf d}(x(t))\le \beta({\bf d}(z),t) \qquad \forall t\in[0,+\infty[,
\eeq
so the theorem is proved.
\end{proof}

 We conclude this section with   an  application of  Theorem \ref{stimaenergiaintro} to Mechanics. 
\vskip0.3truecm
\subsection{The gyroscope: controlling the  nutation through  precession and spin}\label{gyroex}

$\,$
\vskip0.3truecm
A gyroscope can be represented as  a mechanism composed by a rotor --in our setting a spinning disk--  and two gimbals. The spin axis of the rotor is 
fixed to the inner gimbal, whose spin axis is fixed to the outer gimbal (see Figure \ref{gyro2}).

\begin{figure}[h!]
 \hskip-1.5cm\includegraphics[scale=0.7]{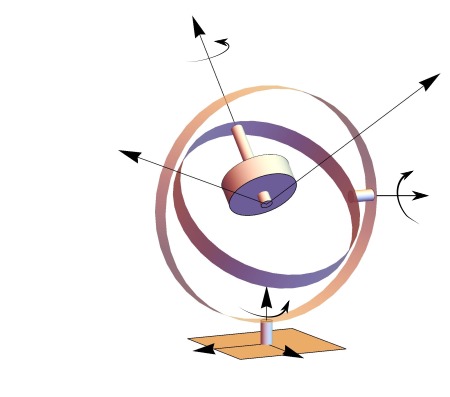}
 \put(-100,30){$0$}
 \put(-22,132){$ \theta$}
 \put(-90,14){$Y$}
 \put(-150,17){$X$}
 \put(-185,142){$x$}
 \put(-5,180){$y$}
 \put(-142,212){$z$}
 \put(-117,52){$Z$}
 \put(-125,195){$ \psi$}
 \put(-92,50){$ \phi$}
 
 \caption{\label{gyro2}}
\end{figure}

Besides an inertial  reference frame $OXYZ$  we consider a reference frame $oxyz$ fixed to  the rotor.  In particular, we choose the latter reference so that 
 the centre of mass of the rotor has coordinates $(0,0,z_G)$. The motion of the rotor can be parametrized by  Euler angles  as depicted 
in Figure \ref{gyro2}:  the outer gimbal's position is represented  by the \emph{precession} angle $\phi$, 
 the  inner  gimbal's position  is given by the \emph{nutation} angle $\theta$, and the  rotor's position  is measured  by the \emph{spin} angle $\psi$. 
The  kinetic energy (in the inertial frame) is so  given by
$$
{\mathcal T}=\frac{1}{2}I_0(\dot \phi^2\sin^2 \theta+ \dot \theta^2)+\frac{1}{2} I(\dot \phi \cos\theta + \dot\psi)^2,
$$
where $I_0$ is  the moment  of inertia of the rotor with respect to any  axis through $o$ and orthogonal to $z$  \footnote{All these moments coincide because of the symmetry of the rotor.} and $I$ is the moment of inertia of the rotor about its spin axis $oz$. We have  tacitly assumed that the rotor's mass $M$ is the only non-negligible mass of the system.    For simplicity, we also  suppose $I_0=I$.
If $g$ denotes  the gravitational acceleration,   the  potential  energy $\mathcal V $   is given by 
$$
\mathcal V(\theta):=M gz_G \cos\theta \quad \forall \theta\in[-\pi/2,\pi/2].
$$
We will  regard the precession velocity $\dot \phi$ and the spin   velocity $\dot \psi$   as {\it controls} belonging to $U=\RR^2$. Considering the predetermination of $\phi(\cdot) $ and $\psi(\cdot)$ as a  holonomic constraint,  we assume the classical  D'Alembert hypothesis (see \cite{aldobressan}). 

The resulting control mechanical system is   
\begin{equation}\label{gyrosys}
\begin{cases}
 \displaystyle{\dot\theta=\frac{1}{I}\pi_{\theta}}\\
\displaystyle{\dot \pi_{\theta}=Mgz_G\sin\theta-I\sin\theta \dot\phi\dot\psi},
\end{cases} 
\end{equation}
where  $\pi_\theta$ is the conjugate momentum  $\pi_\theta:=\frac{\partial ({\mathcal T}+\mathcal V)}{\partial\dot\theta}=I\,\dot\theta$.

If we set  $u:=(\dot \phi,\dot \psi)$, $x=(x_1,x_2)^{tr} :=(\theta,\pi_\theta)$ , $f_0(x)=(I^{-1}x_2,Mgz_G\sin x_1)^{tr}$, and $f_{11}(x)=(0,-I\sin x_1)^{tr}$ we obtain the control-quadratic control system
\begin{equation}\label{balgyro}
\dot x=f(x,u) :=f_0(x)+u_1u_2 f_{11}(x) , 
\end{equation}
with $(u_1,u_2)\in\RR^2$. 
 The state space of the control system \eqref{balgyro} is the open set  $\Omega=]-\pi/2,\pi/2[\times \R$ and we choose  $\T=\{(0,0)\}$ as a target and
$l(x_1,x_2)=x_2^2$ as a running cost \,.
\vskip0.4truecm

 Let us set  
 $$W(x_1,x_2):=W_1(x_1,x_2)(2-|W_2(x_1,x_2)|),$$
 where 
 \begin{align*}
&W_1(x_1,x_2):=\tan^2x_1+x_2^2,\\
  &W_2(x_1,x_2):=\begin{cases}
              \sin\left(2\text{arctan}\left(\frac{-\tan x_1+\sqrt{3} x_2}{\sqrt{3}\tan x_1+x_2}\right)\right)&\text{ if } x_2\not=-\sqrt{3}\tan x_1\\
              0&\text{ otherwise }.
             \end{cases}
 \end{align*}

With some computation,  one proves that
 \begin{Claim}\label{claimgyro}  For any  $\p< \min\{1/I,8\sqrt{3}/3\}$, the function $W$ is $p_0$-MRF for the problem $(f,l, \T)$.
\end{Claim}
   
 Therefore, by Theorem \ref{stimaenergiaintro} we can conclude that the control system for the nutation $\theta$ and its conjugate moment $\pi_\theta$
 is \emph{GAC} to the origin. In addition, the optimal value $V$ of the minimum problem with running  cost equal to $\pi_\theta^2$ \ $(=I^2\dot\theta^2)$ verifies 
$$V(\bar \theta, \bar\pi_\theta)\leq  \frac{W(\bar\theta, \bar\pi_\theta)}{p_{0}}  $$ for all initial data 
$(\bar \theta, \bar\pi_\theta)$ and $\p<\min \{1/I,8\sqrt{3}/3\}$. 
Notice that, as it might be expected, the larger the  moment of inertia $I$ is,
the larger is the  provided  bound for $V$.

\vskip0.3truecm
\section{The rescaled problem}\label{rescsec}

The main step of the proof of Theorem \ref{stimaenergiaintro} is based on Theorem \ref{nuovo} below, which concerns  GAC and optimization  for a  cost-dynamics pair $({\bf l},{\mathbf{f}})$  verifying the following boundedness and  uniform continuity hypothesis:
\vskip0.3truecm
{\bf Hypothesis A$_{UC}$}\emph{ The vector field $({\bf l},{\mathbf{f}})$ is continuous on $( \Omega \backslash \T)\times U$ and, for every compact subset ${\mathcal K}\subset \Omega \backslash \T$, it   is bounded and  uniformly continuous on ${\mathcal K}\times U$.}
\vskip0.3truecm
We point out that the control set $U$ is still allowed to be unbounded.
\vskip0.3truecm

Let us consider the exit time optimal control problem
\bel{sis1}
y' = {\mathbf{f}}(y,v),  \quad y(0) = z, 
\eeq
\bel{cost1}
 {\bf V} (z) :=\inf_{(y,v)\in  \mathcal{ A}_{{\bf f}}(z) }\int_0^{T_y}\mathbf{l}(y(t),v(t)) dt.
\eeq
  
\begin{theorem}\label{nuovo}
Let us assume Hypothesis {\rm \bf A$_{UC}$}, and
let $W$ be   a $\p$-Minimum Restraint Function for the problem  $({\bf l},{\mathbf{f}},\T)$.  Then:
 \begin{itemize}
 \item[{\bf (i)}]  system {\rm \eqref{sis1}} is  GAC to $\T$;
  \item[{\bf (ii)}] moreover, if    $\p>0$,
  \bel{}
 {\bf V} (z) \leq  \frac{W(z)}{{{\p}}}\,  \qquad\forall z\in \Omega\setminus \T.
\eeq
\end{itemize}
 \end{theorem}
\vskip0.3truecm
\subsection{Preliminary results}
$\,$
\vskip0.3truecm

  The proof of Theorem \ref{nuovo}  relies on Propositions \ref{claim1}, \ref{cB}, and \ref{claim2bis}  below. 
 Hypothesis  {\bf A$_{UC}$} is used throughout the whole subsection.
   \begin{proposition}\label{claim1}   
 For every $\sigma>0$ there exists a 
  continuous, increasing   map $\gamma:]0,2\sigma]\to ]0,+\infty[$ such that, for every $r\in]0,2\sigma]$,  
  \bel{c0}
  H_{{\bf l},{\mathbf{f}}} (x,\p,D^*W(x)) <-\gamma(r) \qquad \forall x\in W^{-1}([r,2\sigma])\, \text{ and } \,  p\in D^*W(x).
  \eeq
\end{proposition}
This result  is a consequence of the upper semicontinuity of the set-valued map $x\to D^*W(x)$   together with the continuity of $(x,p)\mapsto H_{{\bf l},{\mathbf{f}}}$, when the  latter is restricted to  the sets $W^{-1}([r,2\sigma])\times \R^n$  (for the details, see \cite[Proposition 3.1]{MR13}).

\vv
\begin{proposition}\label{cB}  
For a given  $\sigma>0$, let $\gamma(\cdot)$ be  a map   as in  Proposition \ref{claim1}. Then there exists a continuous, decreasing function $N: ]0,2\sigma]\to ]0,+\infty[$ such that,  setting
 $$
 H_{{\bf l},{\mathbf{f}}, N(r)}(x,p_0,p):=\min_{u\in U\cap B(0,N(r))}\Big\{\langle p, {\mathbf{f}}(x,u)\rangle+\p {\bf l}(x,u) \Big\} \quad \forall r\in ]0,2\sigma],
 $$
 we get
\bel{c2'}
 H_{{\bf l},{\mathbf{f}}, N(W(x))} (x,p_0,D^*W(x))< -\gamma(W(x)) \qquad \forall x\in W^{-1}(]0,2\sigma]).
\eeq
 \end{proposition}
 \begin{proof} Given $r\in]0,2\sigma]$, let us first show that  there exists some $N(r)$ such that 
 \bel{c0'}
  H_{{\bf l},{\mathbf{f}}, N(r)} (x,\p,D^*W(x))<-\gamma(r)<0 \qquad \forall x\in W^{-1}([r,2\sigma])\, \text{ and } \,  p\in D^*W(x).
  \eeq
Assume by contradiction that for any integer $k$ there is some pair $(x_k,p_k)$ with $x_k\in W^{-1}([r,2\sigma])$ and $p_k\in D^*W(x_k)$ such that,  
\bel{fk}
\Big( u \in U: \quad \langle p_k,{\mathbf{f}}(x_k,u )\rangle+\p {\mathbf{l}}(x_k,u ) <-\gamma(r)<0\Big) \ \Longrightarrow \ |u|>k 
\eeq
  (by Proposition \ref{claim1}, controls verifying the inequality surely exist).  Because of the compactness of $W^{-1}([r,2\sigma])$ and of the upper semicontinuity of the set-valued map $D^*W(\cdot)$, there is a subsequence, which we still denote  $(x_k,p_k)$, converging to some $(\bar x,\bar p)$ such that $\bar x\in W^{-1}([r,2\sigma])$ and $\bar p\in D^*W(\bar x)$.  Since $W$  verifies (\ref{c0}), there is some $\bar u\in U$   such that
 $$
\alpha:=\langle \bar p, {\mathbf{f}}(\bar x, \bar u)\rangle+\p{\mathbf{l}}(\bar x, \bar u)<-\gamma(r)<0.
 $$
Thus,   the uniform continuity of  the maps ${\bf l}$, ${\mathbf{f}}$  on $W^{-1}([r,2\sigma])\times U$  implies that  
 $$
\langle (p_k, {\mathbf{f}}(x_k,\bar u)\rangle+\p {\mathbf{l}}(x_k,\bar u)+\gamma(r)<\frac{\alpha}{2}<0 \qquad \forall k\ge \bar k,
 $$
 some integer $\bar k$, which contradicts   \eqref{fk} as soon as $k>|\bar u|$.
 
\noindent  Moreover, for every $r_1,r_2\in ]0,2\sigma]$, $r_1< r_2$,  one clearly has $N(r_1)\ge N(r_2)$ and, enlarging $N(r)$ if necessary, one can assume  the map $r\mapsto N(r)$ continuous. Therefore, for any  $x\in W^{-1}(]0,2\sigma])$,  the thesis (\ref{c2'}) follows from (\ref{c0'}) as soon as $r= W(x)$.
 \end{proof}
 
 \vv
Let us introduce the following definition, useful in the sequel. 
 \begin{definition}\label{FDBK}  
 Let $\sigma>0$  and fix a selection $p(x)\in D^*W(x)$ for any  $x\in W^{-1}(]0,2\sigma])$.   Let $\gamma(\cdot)$, $N(\cdot)$ be  the same as in  Proposition \ref{cB}.   We call a {\rm feedback} on $W^{-1}(]0,2\sigma])$ a map  
 $$
 x\mapsto {\bf u}(x)\in U\cap B(0,N(W(x))
 $$ 
  verifying
\bel{feed1}
\langle p(x), {\mathbf{f}}(x,{\bf u}(x))\rangle+\p {\mathbf{l}}(x,{\bf u}(x)) <-\gamma(W(x)) 
\eeq
for every $x\in W^{-1}(]0,2\sigma])$.
\end{definition}
 
Moreover, for any $\mu>0$ and any continuous path $\tilde y:[\tau, +\infty[\to\R^n$ such that  $W(\tilde y(\tau))>\mu$,  we  define the time to reach the enlarged target $W^{-1}([0,\mu])$ as
\bel{Tz}
{\mathcal T}_{\tilde y}^\mu\, :=\inf\{r\ge\tau: \ W(\tilde y(r))\le \mu\}
\eeq
(in particular, ${\mathcal T}_{\tilde y}^\mu=+\infty$ if $W(\tilde y(r))> \mu$ for all $r\ge\tau$).
\vv
\begin{proposition}\label{claim2bis}  
 Fix  $\sigma\in]0,W_0[$, and let $\gamma(\cdot)$, $N(\cdot)$ be  as in  Propositions \ref{claim1}, \ref{cB}. Moreover, let $\varepsilon$, $\bar\mu$,  $\hat\mu$  verify $\varepsilon>0$ and $0<\hat\mu<\bar\mu\le\sigma$.  Then there exists some $\delta>0$ such that,  for every  partition $\pi=(t^j)$ of $[0,+\infty[$ with diam$(\pi)\le\delta$ \footnote{ A  {\it partition}  of $[0,+\infty[$ is a sequence $\pi=(t^j) $ such that
$t^0=0, \quad t^{j-1}<t^j$ \, $\forall j\ge 1$, and
   $\lim_{j\to+\infty}t^j=+\infty$.  The number 
    diam$(\pi)\doteq\sup(t^{j }-t^{j-1})$
   is called the {\rm diameter}  of the sequence $\pi$.} and  for each $x\in\Omega\setminus\T$ satisfying $W(x)=\bar\mu$,   there are a piecewise constant control $v:[0,\hat t]\to  U\cap B(0,N(\hat\mu))$ and a  solution $y:[0,\hat t]\to W^{-1}([\hat\mu,\bar\mu])$ to the Cauchy problem
$$
y'=  {\mathbf{f}}(y,v), \qquad y(0)=x,
$$
enjoying following properties:
\begin{itemize}
\item[\bf (a)]  $\hat t:={\mathcal T}^{\hat\mu}_y<+\infty$ and $\bar n:=\sup\{j\ge1: t^{j-1}< {\mathcal T}^{\hat\mu}_y\}<+\infty$.
 \item[\bf (b)] 
  for every  $t\in[0,\hat t[$ and $j\ge1$   such that  $t\in[t^{j-1}, t^j[$,
\bel{dainserire}
  W(y(t))-W(y(t^{j-1}))+\p  \int_{t^{j-1}}^t \mathbf{l}(y(\tau),v(\tau))\,d\tau  \le -\frac{\gamma(W(y(t^{j-1}))) }{\varepsilon+1}(t -t^{j-1}).
\eeq
\end{itemize}
 \end{proposition}
 \begin{proof} 
 Let $p(\cdot)$ be a selection of $D^*W$ on $W^{-1}({[{\hat\mu/4},{2\sigma}]})$  and let us consider a feedback ${\bf u}$ as in Definition  \ref{FDBK}.   Let $M$  denote the  sup-norm of ${\mathbf{f}}$  on $W^{-1}({[{\hat\mu/4},{2\sigma}]})\times U$, and 
let $\omega_{{\bf l}}(\cdot)$ be the modulus of continuity of ${\bf l}$ on $W^{-1}({[{\hat\mu/4},{2\sigma}]})\times U$.
By the local semiconcavity and the properness of  $W$,  Lemma \ref{Lscv} implies that  there exist
  $\rho$, $L>0$ such that,  for any $x$ belonging to the compact set $W^{-1}({[{\hat\mu/4},{2\sigma}]})$,  one has \footnote{The inequality (\ref{scv}) is usually formulated with the proximal superdifferential  $\partial^P F$ instead of $\partial_C F$. However, this does not make a difference here since $\partial^P F=\partial_C F$ as soon as $F$ is locally semiconcave.}
\bel{scv}
W(\hat x)-W(x)\le  \langle p,\hat x-x \rangle+\rho\,|\hat x-x|^2 \qquad \forall p\in D^*W(x),
\eeq
for every $\hat x$ such that the segment  $[x,\hat x]\subset W^{-1}({[{\hat\mu/4},{2\sigma}]})$, and
\bel{Lip}
 |p|\le L \qquad \forall p\in D^*W(x).
\eeq
Let   $\psi:\R^n\to[0,1]$ be a $C^{\infty}$ (cut-off) map such that
\bel{psi}
 \psi = 1 \quad\hbox{on}\quad W^{-1}([{\hat\mu/2}, \sigma]) , \qquad \psi = 0
 \quad\hbox{on}\,\,\, \R^n\backslash W^{-1}([{\hat\mu/4}, {2\sigma}])\,.
\eeq
Let $\omega$ denote the modulus of continuity of the product $(\psi\,{\mathbf{f}})$ on $\R^n\times U$.

\noindent We set
\bel{E1}
\delta:=\min\left\{  \frac{\hat\mu}{2LM},  \delta_2\right\},
\eeq
where   $\delta_2>0$   verifies 
\bel{E2}
\frac{L\,\omega  \left(
M\, \delta_2\right) +
\rho\,  M^2 \, \delta_2 +\p\,\omega_{{\bf l} } \left(
 M\,\delta_2\right)}{\gamma(\hat\mu/4)}= \frac{\varepsilon}{\varepsilon+1}.
\eeq
Let $ \pi=(t^j)$ be an arbitrary partition of $[0,+\infty[$ such that  diam$(\pi)\le \delta$.  For each $x\in\Omega\setminus\T$  verifying $U(x)=\bar\mu$,  define recursively a sequence of trajectory-control  pairs $(y^j,v^j):[t^{j-1},t^j]\to \Omega\times U$, $j\ge1$,
as follows:
\begin{itemize}
\item  $y^1(t^0):= x^1:= x\, , \ \ v^1:= {\bf u}(x^1);$
 \item  for every  $j> 1$,
$$
y^{j }(t^{j-1}):= y^{j-1}(t^{j-1}):= x^j\,, \quad v^j:= {\bf u}(x^j);
$$
 \item for every $j\ge 1$,   $y^j:[t^{j-1},t^j]\to\R^n$ is a solution of the Cauchy problem
$$
y'(t) = \psi(y)\,{\mathbf{f}}(y,v^j)  \quad
y (t^{j-1}) = x^j.
$$
\end{itemize}
Notice that, by the continuity of the vector field and because of the cut-off factor $\psi$,  any trajectory $y^j(\cdot)$ exists globally  and cannot exit the compact subset $W^{-1}({[{\hat\mu/4},{2\sigma}]})$. 
Let us set
$$
(y(t), v(t)):=(y^j(t), v^j) \ \ \forall t\in[t^{j-1},t^j[, \quad \text{for every $j\ge1$.}
$$
In view of  the $L$-Lipschitz continuity of $W$ on $W^{-1}({[{\hat\mu/4},{2\sigma}]})$, the condition $\delta\le \hat\mu/2LM$ in (\ref{E1}),    implies that
  $ |W(y^j(t))- W(x^j)|\le L|y^j(t)- x^j|\le \hat\mu/2,$
  so that 
  $$
  W(y^j(t))\ge \hat\mu/2   \quad \forall t\in [t^{j-1}, t^j], \quad \text{for every $j\ge1$,}  
  $$
  as soon as  $W(x^j)\ge \hat\mu$.

\noindent Recalling that $|\psi|\le 1$ and  $\psi( x^j)=1$ when $x^j\in W^{-1}([\hat\mu/2,2\sigma])$,   (\ref{feed1}) and  (\ref{scv}) and  imply that, for every  $j\ge1$ such that $t^{j-1}< {\mathcal T}^{\hat\mu}_y$ (see Definition \ref{Tz}),     one has, $\forall t\in [t^{j-1}, t^j]$,
 \begin{align*}
  &W(y^j(t))-W(x^j) +\p\int_{t^{j-1}}^t \mathbf{l}(y^j(\tau),v^j)\,d\tau\le  \langle p(x^j),y^j(t)- x^j\rangle+\rho|y^j(t)- x^j|^2 +\\
& \p\int_{t^{j-1}}^t  \left[ \mathbf{l}(y^j(\tau),v^j)- \mathbf{l}(x^j,v^j)\right]\,d\tau+ \p\,  {\bf l} (x^j,v^j)(t-t^{j-1})\\
&\le  \left\langle p(x^j),\int_{t^{j-1}}^{t }\left[\psi(y^j(\tau)) \, {\mathbf{f}}(y^j(\tau),v^j)-{\mathbf{f}}(x^{j},v^j )\right]\,d\tau\right\rangle 
\\&+ \rho\left( \int_{t^{j-1}}^{t}\left|\psi(y^j(\tau)){\mathbf{f}}(y^j(\tau),v^j)\right|\,d\tau\right) ^2  +\p\, \omega_{{\bf l} } \left(M   \, (t^j-t^{j-1})\right)\, (t-t^{j-1})\\ 
&+  \left\langle p(x^j),{\mathbf{f}} (x^{j},v^j) \right\rangle\, (t-t^{j-1}) +\p\, \mathbf{l}(x^j,v^j)(t-t^{j-1})\\
\le &~L\,\omega  \left( M \, (t^j-t^{j-1})\right)\, (t-t^{j-1})+
\rho\, M^2  \, (t-t^{j-1})^2 \\
&+\p\, \omega_{{\bf l} } \left(M   \, (t^j-t^{j-1})\right)\, (t-t^{j-1})-
   \gamma(W(x^j))(t-t^{j-1})   \\
 \le &\left[\frac{L\,\omega  \left(
M \, (t^j-t^{j-1})\right) +
\rho\,M^2 \, (t^j-t^{j-1})+ \p\,\omega_{{\bf l} } \left(
 M \, (t^j-t^{j-1})\right)}{ \gamma(W(x^j))}-1\right] \\
 &\cdot \gamma(W(x^j))( t-t^{j-1}).
\end{align*}
Since $ \forall t\in[t^{j-1},t^j]$,  $t-t^{j-1}\le \delta \le \delta_2$,  by (\ref{E2})  it follows that
\bel{stimaE}
W(y^j(t))-W(x^j)+\p\int_{t^{j-1}}^t \mathbf{l}(y^j(\tau),v^j)\,d\tau\le -\frac{    \gamma(W(x^j))}{\varepsilon+1}(t -t^{j-1}),
\eeq
  which implies, also recalling the definition $x^j= y^{j-1}(t^{j-1})$,
\bel{stimaE2}
\begin{split}
 W(y(t))-W(x)&+\p\int_0^t \mathbf{l}(y(\tau),v(\tau))\,d\tau\\
 &=[W(y^j(t))-W(x^j)]+\dots + [W(y^1(t^1))-W(x)] \\
 & + \p\int_{t^{j-1}}^t \mathbf{l}(y^j(\tau),v^j)\,d\tau +\dots +\p\int_{0}^{t^1} \mathbf{l}(y^1j(\tau),v^1)\,d\tau\\
 & \le  -\frac{\gamma(W(x^j))(t-t^{j-1})+   \sum_{i=1}^{j-1} \gamma(W(x^i))(t^i-t^{i-1})}{\varepsilon+1}.
\end{split}
\eeq

 In particular, (\ref{stimaE2})   yields   that $W(y(t))\le\ W(x)=\bar\mu$ for all $t\in[0,t^j]$.

\vv
Notice that ${\mathcal T}^{\hat\mu}_y<+\infty$.  Indeed,
   if   by contradiction  ${\mathcal T}^{\hat\mu}_y=+\infty$,  (\ref{stimaE2}) held true for all $t\in[0,t^j]$ with   $j$ arbitrarily large, i.e. (since $(t^j)$ is  a partition of $[0,+\infty[$),  for all $t\ge0$. Therefore,   recalling that $\gamma(W(x^i))\ge \gamma(\hat\mu/4)>0$ for all $i=1,\dots,j$,  one would have $\lim_{t\to+\infty} W(y(t))= 0$, 
   which is not allowed, since,  by the definition of ${\mathcal T}^{\hat\mu}_y$,
   \bel{fuori}W(y(t))>\hat\mu\qquad \forall t\in[0,{\mathcal T}^{\hat\mu}_y[ .\eeq   Let us set  
 $$
 \hat t:= {\mathcal T}^{\hat\mu}_y (<+\infty),
 $$
 so that $\bar n$ reads
 $$
 \quad  \bar n=\sup\{j\ge1: t^{j-1}< \hat t\}.
 $$
 Let us observe that $\bar n<+\infty$.   
 Finally, notice that, because of (\ref{fuori}), $\psi(y(t) )= 1$ for every $t\in[0, t^{\bar n}]$. Hence,  for any $j\in\{1,\dots,\bar n\}$,  $y^j(\cdot)$ is  a solution of
 $$
 \frac{dy}{dt} = {\mathbf{f}}(y,v^j)  \ \ \forall t\in[t^{j-1},t^j], \quad
y (t^{j-1}) = x^j.
$$
It follows that conditions  {\bf (a)}--{\bf (b)}  are  satisfied.
\end{proof}

\vv
 \subsection{Proof of Theorem  \ref{nuovo}}
 
 $\,$
\vskip0.3truecm
 Let   $\sigma\in]0, W_0[$   and  let $\gamma(\cdot)$, $N(\cdot)$ be defined as in Proposition \ref{cB}. Fix $\varepsilon>0$ and let $(\nu_k)\subset]0,1]$ be   a sequence   such that $1=\nu_0>\nu_1>\nu_2>\dots$ and $\lim_{k\to\infty}\nu_k=0$.   Assume that $z\in W^{-1}(]0,\sigma])$ and set
$$
\mu_k:= \nu_k W(z) \quad \forall k\ge 0.
$$
We are  going to exploit   Proposition \ref{claim2bis} in order to build  a trajectory-control pair
$$
(y,v):[0,\bar t[\to (\Omega\setminus\T)\times U
$$
by   concatenation 
$$
(y(t),v(t)) = (y_k(t),v_k(t)) \quad  \forall t\in [t_{k-1},t_k[, \quad \forall k\ge1,
$$
where the pairs $ (y_k(t),v_k(t))$ are described by induction as follows.
\vv

{\it The case  $k=1$.} Let us begin by constructing  $(y_1,v_1)$.  Let us set    $\bar\mu=\mu_0$,  $\hat\mu=\mu_1$,  and let us build a trajectory-control pair    
$$
(y_1,v_1):[0,\hat t]\to W^{-1}([\mu_1,\mu_0])\times U\cap B(0,N(\mu_1)), \qquad y_1(0)=z,
$$ 
 according to  Proposition \ref{claim2bis}.  We set $t_0:= 0$ and $t_1:= \hat t$ and   observe that,  in view of {\bf (a)} in Proposition \ref{claim2bis}, $t_1= {\mathcal T}_{y_1}^{\mu_1}$.

{\it The case  $k>1$.}  Let us define $(y_k,v_k)$ for $k> 1$.  Let us set     $\bar\mu=\mu_{k-1 }$,  $\hat\mu=\mu_{k }$,  and  construct
$$
(\hat y_k,\hat v_k):[0,\hat t]\to W^{-1}([\mu_k,\mu_{k-1}])\times U\cap B(0,N(\mu_k)), \qquad \hat y_k(0)=y_{k-1}(t_{k-1}),
$$
still  according to  Proposition \ref{claim2bis}. We set  $t_k:= t_{k-1}+\hat t$ and  $(y_k,v_k)(t)=(\hat y_k,\hat v_k)(t-t_{k-1})$ $\forall t\in[t_{k-1},t_k]$. We observe that $t_k=  {\mathcal T}_{y_k}^{\mu_k}$.

The concatenation procedure is concluded as soon as we set   $\bar t :=  \lim_{k\to \infty} t_k$. Notice that it may well happen that $\bar t=+\infty$. 

 We claim that
\bel{raggiunge}
\lim_{t\to\bar t^-}  {\bf d}(y(t)) = 0.
\eeq
 Indeed, for every $k\ge1$,    Proposition \ref{claim2bis}  yields the existence of a finite partition $\pi_k=\{\hat t^0_k,\dots,\hat t^{\bar n_k}_k\}$ of $[0,t_k- t_{k-1}]$ such that, setting,
$$t_k^j:=  t_{k-1}+\hat t^j_k  \qquad\forall j\in\{0,\dots, \bar n_k\},
$$
one has      $y(0)\,(=y_1(0))=z$,  and,   for every $k\ge1$:
\begin{itemize}
\item[\bf (a)$_k$]
$y_{k+1}(t_{k}) = y_{k}(t_{k})$, \, $W(y_k(t_{k-1})) = \mu_{k-1 }$; and
 \newline   $W(y_k (t_k))<W(y_k(t))\le W(y_k(t_{k-1}))\le W(z)$   \,  $\forall t\in[t_{k-1} ,t_k[$;
\item[\bf (b)$_k$]  for all $j\in\{1,\dots, \bar n_k\}$, \newline
$W(y_k (t))-W(y_k(t_k^{j-1})) +\p\int_{t_k^{j-1}}^t \mathbf{l}(y_k^j(\tau),v_k(\tau))\,d\tau  \le$ \newline
\hphantom{mmmmmmmmmmmmmm} $  -\frac{1}{\varepsilon+1} \gamma(W(y_k(t_k^{j-1})))(t-t_k^{j-1}) $   \, $\forall   t\in[t_k^{j-1}, t_k^j[$.
 \end{itemize}
In particular, by  {\bf (a)$_k$}, claim (\ref{raggiunge}) is equivalent to
\bel{discrete}
 \lim_{k\to\infty} {\bf d}(y_k(t_{k}))=0.
\eeq
Since  $W$ is proper and  positive definite,  (\ref{discrete}) is a straightforward consequence  of
$$\lim_{k\to\infty} W(y_k(t_k)) = \lim_{k\to\infty} \nu_{k }\,W(z) = 0,$$
so   (\ref{raggiunge}) is verified as well.

\vv
  We now need precise estimates of both the decreasing rate of $W$ and  the cost gain along $(y,v)$.   
  
  Let us consider $t$, $k$, $j$ such that $t<\bar t$ and  $t\in[t_{k}^{j-1}, t_k^j[$.  Notice  that  {\bf (b)$_k$}  implies  
\bel{sigma1}
W(y(t))\le W(y_k(t_{k}^{j-1}))\le W(y(t_{k-1}))\le\dots\le W(y(t_1))\le W(z)\le\sigma,
\eeq
and, in view of  the definition of $(y_k,v_k)$, also  
\begin{align*}
&W(y_k(t))-W(y_k(t_{k-1})) +\p\int_{t_{k-1}}^t \mathbf{l}(y_k(\tau),v_k(\tau))\,d\tau =  \\
 &[W(y_k(t))-W(y_k(t_k^{j-1}))]+   
 [W(y_k(t_k^{j-1}))-W(y_k(t_k^{j-2}))] +\dots  
  +[W(y_k(t_k^{1}))-W(y_k(t_k^{0}))]\\
&+\p\int_{t_k^{j-1}}^t \mathbf{l}(y_k(\tau),v_k(\tau))\,d\tau+\dots+\p\int_{t_k^{0}}^{t_k^{1}} \mathbf{l}(y_k(\tau),v_k(\tau))\,d\tau\\
& \le-\frac{1}{\varepsilon+1}\left[\gamma(W(y_k(t_k^{j-1})))(t-t_k^{j-1})+\sum_{i=1}^{j-1} \gamma(W(y_k(t_k^{i-1})))(t_k^{i }-t_k^{i-1})\right] .
\end{align*}
 By the monotonicity of $\gamma$ one has   $\gamma(W(y_k(t_k^{j-1})))\le \gamma(W(y_k(t_k^{i-1})))$ for any $i=1,\dots,j-1$, which implies   
  $$
  W(y_k(t))-W(y_k(t_{k-1})) +\p\int_{t_{k-1}}^t \mathbf{l}(y_k(\tau),v_k(\tau))\,d\tau\le  -\frac{1}{\varepsilon+1}  \gamma(W(y_k(t_k^{j-1})))(t-t_{k-1}).
  $$
Hence, recalling the definition of $(y,v)$,  we have 
\begin{align*}
&W(y(t))-W(z)  +\p\int_{0}^t \mathbf{l}(y(\tau),v(\tau))\,d\tau = \\ 
&[W(y(t))-W(y(t_{k-1}))]+[W(y(t_{k-1}))-W(y(t_{k-2}))] +\dots +[W(y(t_1))-W(y(0)]\\  
&+\p\int_{t_{k-1}}^t \mathbf{l}(y_k(\tau),v_k(\tau))\,d\tau+\dots+\p\int_{0}^{t_1} \mathbf{l}(y_k(\tau),v_k(\tau))\,d\tau,
\end{align*}
so,  by using (\ref{sigma1}),  we finally obtain  
\bel{EP}
W(y(t))-W(z)  +\p\int_{0}^t \mathbf{l}(y(\tau),v(\tau))\,d\tau\le    -\frac{1}{\varepsilon+1}  \gamma(W(y_k(t_k^{j-1})))t.
\eeq
This  is the key inequality for proving both claim {\bf (i)} and claim {\bf (ii)} of the theorem.
\vv
As for claim {\bf (i)} --stating that the system is (GAC) to $\T$--,   we have to establish the existence of  a ${\mathcal KL}$ function $\beta$ as in  Definition \ref{(GAC)}. 
  Let   $t$ belong to $[0,\bar t[$. Then $t\in[t_k^{j-1}, t_k^j[$ for some  $k\ge1$ and some $j\in\{0,\dots,{\bar n_k}\}$.  Since $l\ge0$, by \eqref{EP}   we get
\bel{stg}
W(y(\tau))+\frac{ \gamma(W(y(t_k^{j-1}))\,\tau}{\varepsilon+1}\le W(z) \qquad \forall \tau\in[t_k^{j-1},t_k^j].
\eeq
Observe that the function $\tilde \gamma:[0,+\infty[\to[0,+\infty[$  defined  by $\tilde \gamma(r):= \min\{r,\gamma(r)\}$   for all $r\in[0,+\infty[$  is  continuous,  strictly increasing, and $\tilde \gamma(r)>0$ \, $\forall r>0$,  $\tilde \gamma(0)=0$. Then,  taking $\tau=t_k^{j-1}$ in (\ref{stg}), one has 
$$
\tilde \gamma(W(y(t_k^{j-1}))\left[1+\frac{ t_k^{j-1}}{\varepsilon+1}\right]\le W(z),
$$
 so that
$$
 W(y(t))\le  W(y(t_k^{j-1}))\le \tilde \gamma^{-1}\left(\frac{\varepsilon+1}{\varepsilon+1+t_k^{j-1}}\,W(z)\right).
$$
By Proposition \ref{claim2bis} it is not restrictive to assume  $diam(\pi_k)\le1/2$. Therefore  we get
$$
W(y(t))\le \tilde \gamma^{-1}\left(\frac{2( \varepsilon+1)}{\varepsilon+1+t }\,W(z)\right).
$$
Proceeding as usual in the construction of the function $\beta$, we  set
\bel{sigma}
\sigma_-(r):=\min\{r\,,\,\min\{{\bf d}(x): \ W(x)\ge r\}\}, \quad
 \sigma^+(r):=\max\{{\bf d}(x): \ W(x)\le r\}.
\eeq
Clearly,  $\sigma_-$, $\sigma^+:[0,+\infty[\to\R$ are continuous, strictly increasing, unbounded  functions such that $\sigma_-(0)=\sigma^+(0)=0$  and
$$
\forall x\in{W^{-1}([0,\sigma])}: \quad \sigma_-(W(x))\le {\bf d}(x)\le\sigma^+(W(x)).
$$
We now define  $\beta:[0,+\infty[\times[0,+\infty[\to[0,+\infty[$ by setting
\bel{defbeta}
\beta(r,t):= \sigma^+\circ\tilde \gamma^{-1}\left(
\sigma_-^{-1}(r)
\,\frac{2(\varepsilon+1)}{\varepsilon+1+t} \right),
\eeq
so, by straightforward calculations, it follows that ($T_y=\bar t$ and)
$$
{\bf d}(y(t ))\le \beta({\bf d}(z), t ) \qquad \forall t\in[0,T_y[.
$$
By the arbitrariness of $\sigma>0$, this concludes the proof of claim {\bf (i)} of the theorem.

\vv
As for  claim {\bf (ii)}, we now observe that  inequality  (\ref{EP})  implies also
$$
 \int_0^{\bar t}\mathbf{l}(y(t),v(t))\,dt =\lim_{k\to+\infty}\int_0^{t_k}\mathbf{l}(y(t),v(t))\,dt \le  \lim_{k\to+\infty}\frac{W(z)-W(y(t_k))}{\p}= \frac{W(z)}{\p}\,, 
$$
from which   (\ref{Wprop}) follows. \,\qed

\section{Control-polynomial systems}\label{balsect}$\,$

in this section and in the next one we will assume the dynamics $\F$  to be a  polynomial of degree $d\geq 0$ in the control variable $u$:
\bel{polymomium-est}
\begin{array}{c}
\displaystyle \dot x=\P(x,u):= f_0(x)+\sum_{i=1}^d\left(\sum_{\alpha\in\N^m, \, \alpha_1+\dots+\alpha_m=i} u_1^{\alpha_1}\cdots u_m^{\alpha_m} f_{\alpha_1,\dots,\alpha_m}(x)\right) , \quad x(0)=z, \\ \, \\
 V(z) :=\displaystyle\inf_{(x,u)\in {\mathcal{ A}}_{f}({z})}  \int_0^{T_x}  l(x(t),u(t))\, dt.  
\end{array}\eeq
 
We assume the  vector fields $f_0,   f_{\alpha_1,\dots,\alpha_m}$   to be continuous and the controls to  range on the set  $$U_r:= [-r, r]^m,$$ 
 for some $r$,  $0< r\leq +\infty$ (if $r=+\infty$ we mean $U_r:=\R^m$). 

On  the one hand such polynomial structure is of obvious interest for applications. For instance,  in the  example of the gyroscope (Section \ref{gyroex}) the dynamics is   quadratic in  the controls, namely the precession and rotation velocities. Also the  impressive behaviour  of the Kapitza pendulum --where a fast oscillation of the pivot turns an unstable (or even  a non-equilibrium) point into a  stable point--  can be explained by saying that the square 
of the pivot velocity  --regarded as a control-- prevails on gravity. Many other mechanical systems, possibly non-holonomic, 
can be thought as control systems with quadratic dependence on  the inputs, see e.g. \cite{BR10}. 

On the other hand, it is natural to try to exploit the control polynomial dependence 
 for  a careful study of the vectogram's convex hull \footnote{In some classical  literature, as well as in some recent papers,  objects akin to  the convex hull of the image of the vector valued function
that maps  $u\in\R^m$ into the (suitably ordered) sequence of all monomials of $u$ up to the degree $d$, are    referred to as {\it spaces of moments}, see e.g. \cite{Akh65,Ego02,Mez04,PT09,Sho50}.}.

\subsection{Near-control-affine systems}

\,\,$\,$

In this subsection we address the task of representing a control-polynomial system -- actually, its convexification -- by means  of a control-affine  dynamics  like
$$ {\F}_{\text{\it aff}}(x,w):= f_0(x)+\sum_{i=1}^d\left(\sum_{\alpha\in\N^m, \, \alpha_1+\dots+\alpha_m=i} w_{\alpha_1,\dots,\alpha_m} f_{\alpha_1,\dots,\alpha_m}(x)\right).$$
Such a representation  in general does not exist, as it is clear when 
 $\F(x,u)= uf_1(x)+u^2 f_{2}(x)$, \, $u\in\R$. However,  an affine  representation is achievable in the case of {\it near-control-affine } systems, where the only non-zero terms are those corresponding to control monomials such that each  component $u_i$ ($i=1,\dots,m$) has an exponent equal  either $0$ or  a fixed odd positive number $K_i$.   To state precisely the main result, let us give some definitions.

\vv
For every $\alpha\in\N^m$, let us set $c(\alpha):=\#\{\alpha_i\ne0; \ i=1,\dots,m\}$.
\begin{definition}[Near-control-affine systems]\label{defbal} We say that the control-polynomial dynamics $f(x,u)$ in {\normalfont(\ref{polymomium-est})} is {\em near-control-affine} if there exist an $m-$tuple  $K=(K_1,\dots,K_m)$ of  positive odd numbers  and a positive integer
 $\dbar \leq m$ such that
$$\F(x,u):=f_0(x)+\sum_{i=1}^{\dbar}\left( \sum_{\alpha\in\N^m: \ c(\alpha)=i, \   \alpha_1\in\{0, K_1\}, \dots, \alpha_m\in\{0,K_m\}} u_1^{\alpha_1}\cdots u_m^{\alpha_m} f_{\alpha_1,\dots,\alpha_m}(x)\right). $$

\end{definition}
\vv

\begin{remark}\label{rmknormd}{\rm    If the near-control-affine system {\normalfont(\ref{polymomium-est})} is of degree $d$, one obviously has   $\dbar\le   d$.
Moreover, when  $\dbar=m$,  the number  $M$ of non-drift terms  of a near-control-affine system $\F$ verifies  $M\le \sum_{k=1}^m \binom{m}{k}=2^m-1$. Indeed  for every  $k\leq m$, the maximum number of non zero terms of the form $u_1^{\alpha_1}\cdots u_m^{\alpha_m} f_{\alpha_1\cdots\alpha_m}$ with $k$ coefficients $\alpha_i\ne0$ is equal to $\binom{m}{k}$.
}
\end{remark}

\vskip 0.3truecm

For every $r\in]0,+\infty[$ we set 
\begin{equation}\label{rbaldef}
 \bar r:=\frac{1}{M}\min\{r^{j K_i}\mid i=1,\dots,m; ~j=1,\dbar\}
\end{equation}
and
$$
\Wbalr:=[-\bar r,\bar r]^M.
$$ 
In addition, we set 
$$
 \bar U_{+\infty} :=\RR^M.$$ 
\vskip0.3truecm
  Theorem \ref{cnear-control-affine}, where we assume  Hypothesis  {\bf A}$_{b}$ below, establishes that near-control-affine systems can be regarded as control-affine systems with independent control variables.
\vskip0.3truecm

\noindent {\bf Hypothesis  A}$_{b}:$

\begin{enumerate}
\item $f$ is near-control-affine;

\item  for every $x\in \Omega\backslash\T$, the map $l(x,\cdot): U_r \to\R$ is bounded;

\item let us  define the (non-negative, continuous) function
$$\ell(x):=\sup_{u\in U} l(x,u).$$ 
The control set for the  minimum  problems  $(\ell,{\F}_{\text{\it aff}},\T)$  coincides with  $\Wbalr$ .

\end{enumerate}
\vskip0.3truecm

\begin{theorem}\label{cnear-control-affine} Let us assume{ Hypothesis} {\bf A}$_{b}$ and let $W$  be a $\p$-MRF for the affine problem  $( \ell,{\F}_{\text{\it aff}},\T)$ for some $\p\ge0$.
 Then the map  $W$ is  a $\p$-MRF for the original (non-affine)  problem $(l,\F,\T)$ as well. In particular, the control system in {\rm \eqref{polymomium-est}}   is \emph{GAC}  to $\T$ and, if  $\p>0$,
$$V(z)\leq \frac{W(z)}{\p} \qquad\forall z\in \Omega\backslash \T.$$
\end{theorem}
\begin{proof}
Let $x\in \Omega\setminus\T$. By assumption one has
$$
\inf_{w\in \Wbalr}  	\Big\{\Big\langle p\,,\,   {\F}_{\text{\it aff}}(x,w) \Big\rangle\Big\}+\p\ell(x) < 0 \qquad \text{for all }p\in D^*W(x).
$$
By Lemma \ref{lnear-control-affine} below, ${\F}_{\text{\it aff}}(x,\Wbalr)\subseteq co \F(x,\Ubalr)$, which implies
\bel{fco}
\inf_{u\in \Ubalr}  \Big\{\Big\langle p\,,\,  \F(x,u) \Big\rangle\Big\}+\p\ell(x) < 0  \qquad \text{for all }p\in D^*W(x).
\eeq
  This concludes the proof, since \eqref{fco}   yields
$$
\inf_{u\in \Ubalr}  \Big\{\Big\langle p\,,\,\F(x,u) \Big\rangle+\p l(x,u)\Big\} < 0  \qquad \text{for all }p\in D^*W(x).
$$
\end{proof}

\begin{lemma}\label{lnear-control-affine}
 For every $r\in [0,+\infty]$
\begin{equation}\label{comb2}
{\F}_{\text{\it aff}}(x,\Wbalr)\subset co~\F(x,\Ubalr)\quad \forall x\in\Omega\setminus\T.
\end{equation}
\end{lemma}
This  result will be proved in Appendix \ref{proofsec}. 

\begin{remark}\label{rmkbal}{
Besides implying Theorem \ref{cnear-control-affine}, Lemma \ref{lnear-control-affine} gives access to classical results on control-affine  systems for the study of  local controllability  of near-control-affine systems. For instance,
consider the driftless,   near-control-affine system (with $d=8$, $K=(1,3,5)$ and $\dbar =2$)
 \bel{balsyst}\dot x=\F(x,u)=u_1u_2^3 f_{1,3,0}(x) +u_1u_3^5 f_{1,0,5}(x) + u_2^3u_3^5 f_{0,3,5}(x),\eeq
 with $x=(x_1,x_2,x_3,x_4)\in \RR^4$, $u=(u_1,u_2,u_3)\in \RR^3$ and
 \begin{equation*}
  f_{1,3,0}(x) =(1,0,x_2,0)^{tr};\quad f_{1,0,5}(x)=(0,1,-x_1,0)^{tr};\quad f_{0,3,5}(x)=(0,0,0,1)^{tr}.
\end{equation*}
Notice that $\{(u_1u_2^3,u_1u_3^5,u_2^3u_3^5)\mid (u_1,u_2,u_3)\in\RR^3\}\subset  \RR^3$ and, for instance,  $$(0,1,1)\notin \{(u_1u_2^3,u_1u_3^5,u_2^3u_3^5)\mid (u_1,u_2,u_3)\in\RR^3\},$$  so $\F$ cannot be parameterized as  control-linear vector field with controls in $\R^3$.
However, by Lemma \ref{lnear-control-affine} the control-linear vector field
 $${\F}_{\text{\it aff}}(x,w)= w_{1,3,0} f_{1,3,0}(x) +w_{1,0,5} f_{1,0,5}(x) +w_{0,3,5} f_{0,3,5}(x)  \quad (w_{1,3,0},w_{1,0,5},w_{0,3,5})\in\RR^3$$
satisfies
$${\F}_{\text{\it aff}}(x,\Wbalr)\subset co(\F(x,\Ubalr)) \quad \forall x\in\RR^4;~\forall r>0.$$ For example, we have that  $f_{1,0,5}(x) + f_{0,3,5}(x) \notin \F(x,\Ubalr)$,  while
$$
  f_{1,0,5}(x) + f_{0,3,5}(x) = \frac 1 2 \F(x,(1,0,2^{1/5})) + \frac 1 2 \F(x,(0,1,2^{1/5})).$$}
\end{remark}

\begin{remark}{\rm
 Let us see a simple utilization of the affine representability of ${\F}_{\text{\it aff}}$ for system \eqref{balsyst}.  Observe that the latter verifies the so-called Lie algebra rank condition, $$Lie_x\{f_{1,3,0},f_{1,0,5},f_{0,3,5}\}=\RR^4 \qquad \forall x\in\RR^4.$$
Indeed the Lie bracket $[f_{1,3,0},f_{1,0,5}]$ coincides with the vector field constantly equal to $(0,0,2,0)^t$, so that $$span\{ f_{1,3,0},f_{1,0,5},f_{0,3,5},[f_{1,3,0},f_{1,0,5}]\} = \R^4$$ at every point. 
Therefore, by Chow-Rashevsky's Theorem the system  $\dot x= {\F}_{\text{\it aff}}(x,w)$   turns out to be small time locally  controllable. Now, by Lemma \ref{lnear-control-affine}
$${\F}_{\text{\it aff}}(x,\Wbalr)\subset co(\F(x,\Ubalr))\quad \forall x\in\Omega\setminus\T.$$ Consequently, by a standard  relaxation argument, we can deduce  that the system
$\dot x= \F(x,u)$ is small time locally controllable as well.
}
\end{remark}

\vskip1truecm
\subsection{Maximal degree weak subsystems}\label{secmaximal}
$\,$
\vskip0.3truecm
In this subsection and the next one,  we assume $r=+\infty$, i.e. $U_r=\RR^m$  and  look for    {\it weak  subsystems}, namely  set-valued selections of the convex-valued multifunction  $x\mapsto co\,\,\F(x,\R^m). $ 

We begin with a class of weak subsystems which we call {\it maximal degree} subsystems.
Theorem \ref{maximalth} below extends in several directions a result contained  in \cite{BR10} and valid for the case $d=2$.
It   states that in order to test if a function $W$ is a $\p$-MRF function for problem \eqref{polymomium-est},  it is sufficient to test  $W$ on the (simpler)   {\it maximal degree}  problem
\bel{max}
\begin{array}{l}
\dot x=\E(x,u), \qquad x(0)=z, \\ \, \\
\displaystyle\inf_{(x,u)\in{\mathcal A}_{\E}(z) }\displaystyle \int_0^{T_{x}} l(x(t),u(t)) dt, 
\end{array}
\eeq
where  the \emph{maximal degree}  control-polynomial vector field $\E$ is defined by
$$\E(x,u):= f_0(x)+\sum_{\alpha\in\N^m, \, \alpha_1+\dots+\alpha_m=d} u_1^{\alpha_1}\cdots u_m^{\alpha_m} f_{\alpha_1,\dots,\alpha_m}(x).$$

We shall  assume the following additional  hypothesis on the running cost:
 \vskip0.5truecm
\noindent
{\bf Hypothesis  A}$_{max}${\it : There exist non negative continuous   functions  
 $M_0=M_0(x)$, $M_1= M_1(x,u)$ such that
\begin{equation}\label{lmax}
l(x,u) =M_0(x) + M_1(x,u),
\end{equation}
with   $M_1$ verifying
$$
M_1(x,0)=0 ,\qquad M_1(x,ku)\leq k^d M_1(x,u) \qquad \forall k\geq 1,~x\in\Omega\setminus\T,~u\in\RR^m.$$}
Notice that running costs   of the form
$$
l(x,u) = l_0(x) + l_1(x)|u| +\dots+ l_d(x)|u|^d,
$$
where the maps $l_i(\cdot)$ are continuous and non-negative, verify Hypothesis
{\bf A$_{max}$}.

\begin{theorem}\label{maximalth} Let us assume{ Hypothesis} {\bf A}$_{max}$, and let $W$ be a $\p$-MRF for the maximal degree  problem $(l,\F^{max}_{\lambda},\T)$,
for some $\p\ge0$. Then the map  $W$ is  a  $\p$-MRF  for the original problem $(l,\F,\T)$.
In particular, the control system in {\rm\eqref{polymomium-est}} is {\rm GAC} to $\T$ and, if  $\p>0$,
$$V(z)  \leq  \frac{W(z)}{\p} \qquad\forall z\in \Omega\backslash \T.$$
\end{theorem}
 
\begin{proof} 
Assume by contradiction that there  exist   $x\in\Omega\backslash\T$ and $p\in D^*W(x)$ such that
\bel{MRFmax1} \p l(x,u) + \langle p, f_0(x)\rangle +\sum_{i=1}^d\left(\sum_{\alpha\in\N^m, \, \alpha_1+\dots+\alpha_m=i}  \langle p, u_1^{\alpha_1}\cdots u_m^{\alpha_m} f_{\alpha_1,\dots,\alpha_m}(x)\rangle \right)\ge 0
\eeq
for all $u\in \R^m$.   By taking $u=0$  we obtain
\bel{MRFmax3}
\p M_0(x) +  \langle p, f_0(x)\rangle \ge 0.
\eeq
By assumption, there exists $\tilde u\in\R^m$ and $\eta >0$  such that
\bel{MRFmax2}
\p\, l(x,\tilde u) + \langle p~,~f_0(x)\rangle + \sum_{\alpha\in\N^m, \, \alpha_1+\dots+\alpha_m=d}  \langle p, \tilde u_1^{\alpha_1}\cdots \tilde u_m^{\alpha_m} f_{\alpha_1,\dots,\alpha_m}(x)\rangle = -\eta .
\eeq
Moreover, \eqref{MRFmax3}-\eqref{MRFmax2} imply
\bel{MRFmax3bis}
 \p k^dM_1(x,\tilde u)+ k^d \sum_{\alpha\in\N^m, \, \alpha_1+\dots+\alpha_m=d}  \langle p, \tilde u_1^{\alpha_1}\cdots \tilde u_m^{\alpha_m} f_{\alpha_1,\dots,\alpha_m}(x)\rangle \leq  -\eta k^d
\eeq
for any $k\geq 0$.
Hence, for every $k\geq 1$
\begin{align*}
  \p l(x,k\tilde u)  +    \langle p~,~ & f_0(x)\rangle +   k \sum_{\alpha\in\N^m, \, \alpha_1+\dots+\alpha_m=1}  \langle p, \tilde u_1^{\alpha_1}\cdots \tilde u_m^{\alpha_m} f_{\alpha_1,\dots,\alpha_m}(x)\rangle+\\
  & \dots+k^{d-1}  \sum_{\alpha\in\N^m, \, \alpha_1+\dots+\alpha_m=d-1}  \langle p, \tilde u_1^{\alpha_1}\cdots \tilde u_m^{\alpha_m} f_{\alpha_1,\dots,\alpha_m}(x)\rangle\\
                        &    + k^d \sum_{\alpha\in\N^m, \, \alpha_1+\dots+\alpha_m=d}  \langle p, \tilde u_1^{\alpha_1}\cdots \tilde u_m^{\alpha_m} f_{\alpha_1,\dots,\alpha_m}(x)\rangle\leq \\
\p k^d M_1(x,\tilde u)   + \p  M_0(x) &+  \langle p~,~ f_0(x)\rangle + k \sum_{\alpha\in\N^m, \, \alpha_1+\dots+\alpha_m=1}  \langle p, \tilde u_1^{\alpha_1}\cdots \tilde u_m^{\alpha_m} f_{\alpha_1,\dots,\alpha_m}(x)\rangle+\\
  & \dots+k^{d-1}  \sum_{\alpha\in\N^m, \, \alpha_1+\dots+\alpha_m=d-1}  \langle p, \tilde u_1^{\alpha_1}\cdots \tilde u_m^{\alpha_m} f_{\alpha_1,\dots,\alpha_m}(x)\rangle\\
                        &    + k^d \sum_{\alpha\in\N^m, \, \alpha_1+\dots+\alpha_m=d}  \langle p, \tilde u_1^{\alpha_1}\cdots \tilde u_m^{\alpha_m} f_{\alpha_1,\dots,\alpha_m}(x)\rangle\leq\\
 \p  M_0(x) +  \langle p~,~ f_0(x)\rangle&+ k \sum_{\alpha\in\N^m, \, \alpha_1+\dots+\alpha_m=1}  \langle p, \tilde u_1^{\alpha_1}\cdots \tilde u_m^{\alpha_m} f_{\alpha_1,\dots,\alpha_m}(x)\rangle+\\
  & \dots+k^{d-1}  \sum_{\alpha\in\N^m, \, \alpha_1+\dots+\alpha_m=d-1}  \langle p, \tilde u_1^{\alpha_1}\cdots \tilde u_m^{\alpha_m} f_{\alpha_1,\dots,\alpha_m}(x)\rangle                         
-\eta k^d .
\end{align*}
If $k$ is  sufficiently large the last term is negative, which contradicts
\eqref{MRFmax1}.
\end{proof}

\begin{remark}{\rm
The thesis of  Theorem \ref{maximalth} cannot be extended to  the case of  bounded control sets.
  For instance, if $d=3$, $n=m=1$, $U=[-1,1]$, $\T=\{0\}$, $l\equiv0$, and $\F(x,u)=(u^2+u^3)x$, one has $\dot x=\F(x,u)\geq 0$
  for $x\geq 0$, so the system is
 not GAC to $\T$ and no control Lyapunov function \footnote{When $l=0$ the notion of $\p$-MRF coincides with that of control Lyapunov function.} exists. However,  $W(x)=x^2$ is a control Lyapunov  function for $(l,\F^{max})$, so   that  the system $\dot x=\F^{max}(x,u)$ is GAC to $\T$.
Nevertheless,  some  symmetry arguments may allow the extension of Theorem \ref{maximalth} to  some special classes of polynomial control systems with {\it bounded}  control sets.
This might be the case when   $d=2$, $U$ is a (compact) symmetric control set (i.e. $u\in U$ implies $-u\in U$) and, for all $x\in\Omega\setminus\T$,   $l(x,\cdot)$ is an even function.
 For  example, consider the system
$$
 \dot x =f(x,u),  \quad  x(0)=z, \qquad u\in U := [-1,1]^2$$
 where \, $$f(x,u):= f_0(x) +u_1 f_{1,0}(x)+u_2 f_{0,1}(x)+ u_1^2 f_{2,0} (x)+ u^2_2 f_{0,2} (x) +  u_1u_2 f_{1,1} (x),$$   
together with the minimum problem
$$
\displaystyle \inf_{(x,u)\in{\mathcal A}_f(z)}\int_0^{T_{x}} (|u| + x^2u^2)dt.
$$
Notice that
$$(l,\F^{max})(x,u)=\frac{1}{2}(l,\F)(x,u)+\frac{1}{2}(l,\F)(x,-u)\in co(l,\F)(x,U)\qquad \forall x\in\Omega\setminus\T,~ u\in U.$$
Therefore, for every $(x,(p_0,p))\in (\Omega\setminus\T)\times\RR^{1+n}$, one has
$$H_{l,\F^{max}}(x,p_0,p)<0\quad \Rightarrow \quad H_{l,\F}(x,p_0,p)<0.$$
Consequently a map  $W$ is  {\rm $\p$-MRF}  for  $(l,\F^{max},\T)$ for some $\p\ge0$  if and only if $W$ is a  {\rm $\p$-MRF} for $(l,\F,\T)$.
Then Theorem \ref{thselectionintro} applies and, consequently,  Theorem \ref{maximalth} can be extended to this case.}
\end{remark}

\vskip0.3cm
\subsection{Diagonal weak subsystems}$\,$\label{secdiagonal}
Another class of weak subsystems is given by the {\it  diagonal  subsystems} described below. We still assume $U=\R^m$. 

Let us use $\mathbf e_1,\cdots,\mathbf e_m$ to denote  the basis of $\R^m$ and  let us set $\mathbf e_0:=0$.
 \begin{definition}\label{diagdef} For every $\lambda$ belonging to the  simplex $\Lambda:=\{\lambda\in\RR^m\mid \sum_{i=1}^m \lambda_i\leq 1;~\lambda_i\geq 0\}$,
\bel{diaginconc}
\F^{diag}_{\lambda}(x,u)
 :=
\sum_{i=0}^m \lambda_i
\F(x,{\lambda_i}^{-\frac 1 d}{u_i}\mathbf e_i) ,
\eeq
where  $\lambda_0:=1-\sum_{i=1}^m \lambda_i$,  
will be called the  {\em $\lambda$-diagonal}  control vector field corresponding to $\F$ and $\lambda$.
\end{definition}

For instance, setting  $f_{\alpha_1,\dots,\alpha_m}:=f_\alpha$ for every $\alpha\in\N^m$,  when $d=2$,  $d=3$  one has
$$
\F^{diag}_{\lambda}(x,u)= f_0(x)+\sum_{i=1}^m \lambda_i^{\frac 1 2}u_i  f_{\mathbf e_i}(x)+\sum_{i=1}^m u_i^2  f_{2\mathbf e_i}(x).
$$
and
$$
\F^{diag}_{\lambda}(x,u)= f_0(x)+\sum_{i=1}^m \lambda_i^{\frac 2 3}u_i  f_{\mathbf e_i}(x)+\sum_{i=1}^m \lambda_i^{\frac 1 3} u_i^2 f_{2\mathbf  e_i}(x)+\sum_{i=1}^m u_i^3  f_{3\mathbf e_i}(x),
$$
respectively.
\begin{remark}\label{remdiag}{\rm 
 Since $\sum_{i=0}^m \lambda_i=1$, this implies that
\bel{Fdiag}
\F^{diag}_\lambda(x,\R^m) \subseteq co~\F(x,\R^m).
\eeq

}\end{remark}
 \vskip0.2truecm
We shall  assume the following   hypothesis on the running cost:
 \vskip0.2truecm
\noindent
{\bf Hypothesis  A}$_{diag}${: There exists a real number $\M\geq0$ such that,  for every $\lambda\in\Lambda$ verifying  $\lambda_i>0$, $i=1,\dots,m$, one has }

\bel{HIP}
l(x,0) + \sum_{i=1}^m \lambda_i l(x,\frac{u_i}{\sqrt[d]{\lambda_i}}\mathbf e_i)\leq \M \,l(x,u)
\qquad \forall u\in\R^m. \eeq
\vskip0.5truecm

\begin{remark}{\rm Notice that for every $q\geq 1$, the particular running cost 
\bel{pollagrang} l(x,u) := l_0(x) + l_{1}(x) |u|+\cdots l_q(x)|u|^q \eeq
does verify Hypothesis {\bf A}$_{diag}$    (with $\M=\sqrt{m}$)\footnote{This is due to the elementary inequalities  
 $$|u_1|+\cdots+|u_m|\leq \sqrt{m} |u|\qquad (|u_1|^q+\cdots+|u_m|^q)^\frac{1}{q}\leq |u|\qquad \forall q>1
.$$}. As a model, simple  case, one could consider $ l(x,u) = |u|^q$, $q\geq d$, so that  the functional to be minimized would be  nothing but  the $q$-th power of the  $L^q$-norm of $u$ .}\end{remark}

\begin{theorem}\label{diagonalth} Assume that  Hypothesis {\bf A}$_{diag}$ holds true for a suitable $M_0\geq 0$,
 and let $W$ be a $\p$-MRF for the $\lambda$-diagonal problem $(l,\F^{diag}_{\lambda},\T)$,  for some $\p\ge0$.
 Then the map  $W$ is  a $\bar\p$-MRF for the original problem $(l,\F,\T)$, where   $\bar\p:= \frac{\p}{\M}$ if $\M>0$, while, if $\M=0$,   $\bar p_0$ is allowed to be any positive real number.
 \\ In particular, the control system  in {\rm\eqref{polymomium-est}} is {\rm GAC} to $\T$ and, if   $\p>0$,
\begin{equation}\label{Wdiag}
V(z)  \leq \frac{\M W(z)}{\p}\qquad\forall z\in \Omega\backslash \T.
\end{equation}
\end{theorem}

\begin{proof}
  Set $\lambda_0=1-\sum_{i=1}^m \lambda_i$  and $\mathbf e_0 = 0$.
First assume $\M >0$. Then for every $i=0,\dots,m$, every $(x,u)\in (\Omega\backslash \T)\times\R^m$ and every $p\in D^*W(x)$, one has
$$
\lambda_i  H_{l,\F}(x, \frac{\p}{K},p) \leq
 \lambda_i
\left< (\frac{\p}{K},p)~,~( l, \F)(x,\lambda_i^{-\frac 1 d}{u_i}\mathbf e_i)\right>
$$
that, summing up for $i=0, \dots,m$, yields
\bel{stimadiag}
\begin{array}{c}
 \displaystyle{H_{l,\F}(x, \frac{\p}{K},p) \leq
\sum_{i=0}^m \lambda_i
\left<(\frac{\p}{K},p)~,~( l,\F)(x,\lambda_i^{-\frac 1 d}{u_i}\mathbf e_i)\right>\leq}\\\,\\
\displaystyle{\frac{\p}{\M } \M  l(x,u) + \left<p~,~ \F^{diag}_\lambda(x,u)\right >
= \p  l(x,u) + \left<p~,~ \F^{diag}_\lambda(x,u)\right > .}
\end{array}
\eeq
Since by hypothesis
$
\max_{p\in D^*W(x)}H_{l,\F^{diag}_\lambda}(x,\p,p) <0,
$
then there exists
$\tilde u$ such that
$$
\p  l(x,\tilde u) + \left< p,\F^{diag}_\lambda(x,\tilde u)\right ><0
\quad \forall p\in D^*W(x),$$
this, together with by \eqref{stimadiag}, implies
$$
H_{l,\F}(x, \frac{\p}{\M },p) <0 \quad \forall p\in D^*W(x)
$$
which indeed is the thesis of the theorem.
Assume otherwise $\M =0$. Then $l\equiv 0$, consequently $W(z)\equiv 0$ and (\ref{Wdiag}) is trivially verified.
Since $W$ is a $\p$-MRF for $(l,\F_\lambda^{diag},\T)$ and since $l\equiv 0$, for every $x\in\Omega\setminus\T$ there exists $\tilde u\in\RR^m$ such that $\langle p,\F^{diag}_\lambda(x,\tilde u)\rangle<0$ for all $p\in D^*W(x)$.
Consequently, for every $\bar p_0\in\RR$ and for every $p\in D^*W(x)$
\begin{align*}
 H_{l,\F}(x,\bar p_0,p)&=\inf_{u\in\RR^m}\langle p,~\F(x,u)\rangle\leq \sum_{i=0}^m\lambda_i\langle p,~ \F(x,\lambda_i^{-\frac 1 d}u_i\mathbf e_i)\rangle\\
 &=\langle p,~ \F^{diag}_\lambda(x,u)\rangle<0.
 \end{align*}
This gives the thesis in the case $\M =0$ and completes the proof.
\end{proof}

\begin{example}\label{diagexample}
{\rm Let $\T:=\{0\}$, $u\in\R^2$ and let us consider  in $\R^2$  the exit-time problem 
\bel{exdiag}
\begin{array}{l}\dot x = \F(x,u):=x +u_1u_2 (|x|^{-1},1)^{tr}   -u_1^2 (1,0)^{tr} -u_2^2 (0,1)^{tr} + 3u_1^2u_2^2 x\quad  x(0)= z;
\\\,\\
V(z) :=   \displaystyle \inf_{(x,u)\in{\mathcal A}(z)} \int_0^{T_{x}} x^2 |u|^2\,dt.  \end{array}
\eeq

Let $\Phi:[0,+\infty[\to \R$ be  a smooth convex function such that $\Phi(0) = 0$ , $\Phi'(0)\geq 1$. In order  to verify that a function of the form
$$
W(x)=\Phi(|x|^2)
$$
is a $\p$-MRF function for some $\p >0$, let us  begin with observing that {\it the  maximal degree subsystem
$$
\dot x = \E(x,u) =x + 3u_1^2u_2^2
x$$
 does not give any useful information}. Indeed
\begin{align*}
H_{l,\E}(x,\p,\nabla W(x))&=\inf_u \left\{\Big\langle \nabla W(x)~ ,~ \E(x,u)\Big\rangle + \p x^2 |u|^2\right\}\\
&=\inf_u \Big\{2 \Phi'(|x|^2) |x|^2 (1+ 3u_1^2u_2^2) + p_\I x^2 |u|^2 \Big\} \geq 0
\end{align*}
for all $x\in\R^2\backslash \{0\}$ and $\p\geq 0$.
On the other hand, by considering the diagonal subsystem
$$
\dot x = \F^{diag}_{(\frac 1 2 ,\frac 1 2)} = x   -u_1^2(1/\sqrt{2},0)^{tr} -u_2^2(0, 1/\sqrt{2})^{tr},
$$
if $\p <1$ \, $ (\leq \Phi'(|x|^2) \text{ for all } x\in \RR^2)$, we get,   for all $x\in \RR^2\setminus\{0\}$, 
$$\begin{array}{c}
H_{l,{ \F^{diag}_{(\frac 1 2,\frac 1 2)}}}(x,\p,\nabla W(x)) \leq \inf_u \Big\{ |x|^2 \Big( \Phi'(|x|^2)(2-u^2) + \p u^2\Big)\Big\} = -\infty,
\end{array}$$
i.e., $W$ is a $\p$-MRF  for the problem   $(l,{ \F^{diag}_{(\frac 1 2,\frac 1 2)}})$. Therefore, in view of Theorem \ref{diagonalth}, $W$ is a $\p$-MRF  for the problem \eqref{exdiag} as well.
}
\end{example}
\appendix

\section{Proof of Lemma \ref{lnear-control-affine}\label{proofsec}} 

For the reader convenience let us recall the statemen of Lemma  \ref{lnear-control-affine}:
\vskip2truemm
{\it 
 For every $r\in [0,+\infty]$
\begin{equation}\label{comb}
{\F}_{\text{\it aff}}(x,\Wbalr)\subset co~\F(x,\Ubalr)\quad \forall x\in\Omega\setminus\T.
\end{equation}
}
\vskip2truemm
We prove  this result in the case all components of the $m$-tuple $K$ are equal to $1$, i.e., $K=(1,\dots,1)$ (this assumption implies $\bar d=m=d$, see Remark \ref{rmknormd}).
Indeed, to prove the theorem when $K$ is a general $m$-tuple  of odd numbers  it is sufficient to  apply the result to the rescaled control-polynomial vector field
$$\hat\F(x,u):=\F(x,u_1^{\frac 1 K_1},\dots,u_m^{\frac 1 K_m}).$$

Fix $k\in\NN$ and denote by $\{1,-1\}^k$ the set of $k$-tuples $(s_1,\dots,s_k)$ with $s_j\in\{-1,1\}$.  
Denote by $P(S)$ the power set of a set $S$ and consider the set-valued map $S_k:\{1,-1\}\to P(\{-1,1\}^k)$  defined by
$$S_k(s)=\left\{(s_1,\dots,s_k)\in\{-1,1\}^k\mid s_1\cdots s_k=s\right\}.$$

Let us begin with a combinatorial result:
\vskip0.4truecm

{\it Claim A:} 
 {\it Let $k,d\in\NN$,   $k< d$. For every  $i_1,\dots, i_k\in\N$, $1\leq i_1<\cdots<i_k\leq d$,   and for every $s\in \{-1,1\}$
\begin{equation}\label{sum}
 \sum_{(s_1,\dots,s_d)\in S_d(s)} s_{i_1}\cdots s_{i_k} =0.
\end{equation}}

 To prove {\it Claim A},  notice that
\begin{equation}\label{zerosum}
\sum_{(s_1,\dots,s_k)\in \{-1,1\}^\k}s_1 s_2\cdots s_k=0.
\end{equation}
Now, fix $i_1,\dots, i_k\in\N$, $1\leq i_1<\cdots<i_k\leq d$ and an auxiliary $k$-uple $\bar{\mathbf  s}=(\bar s_1,\dots,\bar s_k)\in \{-1,1\}^k$.
One has
 \begin{equation*}
\#\left\{(s_1,\dots,s_d)\in \{-1,1\}^d\mid s_{i_h}=\bar s_{h};~h=1,\dots,\k\right\}=2^{d-\k}.
\end{equation*}
Therefore, by a symmetry argument,
\begin{equation}\label{card}
\#\left\{(s_1,\dots,s_d)\in S_d(s)\mid s_{i_h}=\bar s_{h};~h=1,\dots,\k\right\}=2^{d-\k-1} \quad \forall s\in\{-1,1\}.
\end{equation}
In view of (\ref{zerosum}) and of (\ref{card}),  for every $s\in \{-1,1\}$
\begin{equation*}
 \sum_{(s_1,\dots,s_d)\in S_d(s)} s_{i_1}\cdots s_{i_k} = 2^{d-k-1}\left(\sum_{(s_{i_1},\dots,s_{i_\k})\in \{-1,1\}^\k} s_{i_1}\cdots s_{i_k}\right)=0.
\end{equation*}
This concludes the proof of {\it Claim A}.
\vskip0.4truecm

\noindent
We continue the proof of Lemma \ref{lnear-control-affine} by proving  {\it Claim B} below, which  concerns the convex hull $ co~ \F(x,U_r)$.
For every integer $j\ge 1$,  let us set
 $$I_{r,j}:=\begin{cases}
             [-r^j,r^j]&\text{ if } r<+\infty\\
             \RR &\text{ if } r=+\infty.\\
            \end{cases}$$
\vskip0.4truecm
{\it Claim B:}
{\it Let $d\leq m$. For every $k\leq d$,   $i_1,\dots, i_k\in\N$, $1\leq i_1<\cdots<i_k\leq d$, and $w\in I_{r,k}$,
one has
\begin{equation}\label{component}
f_0(x)+w f_{\alpha_1,\dots,\alpha_m}(x)\in co~\F(x,U_r),
\end{equation}
where $\alpha_{j}=1$ for $j\in\{i_1,\dots,i_k\}$ and $\alpha_j=0$ otherwise.}

\vskip0.4truecm
 To prove {\it Claim  B}, denote by $s(w)$ the sign of $w$ and select from $I_{r,1}$  a set of $k$ real numbers $u_{i_1},\dots,u_{i_k}$ such that
$u_{i_1}\cdots u_{i_k} =w$. 

Define     
$$u^{(\mathbf s)}:= \sum_{j=1}^k \mathbf s_{j} |u_{i_j}| \mathbf e_{i_j} \quad \text{for every } \mathbf s:=(s_{1},\dots,s_{k}) \in S_k(s(w)).$$ 
By construction one has  $u^{(\mathbf s)}\in [-r,r]^m=U_r$ and
$$u^{(\mathbf s)}_{i_1}\cdots u^{(\mathbf s)}_{i_k}=w.$$
By {\it Claim A}, for every $h<k$  and every increasing finite subsequence \newline $i_{1}\leq i_{j_1}<\cdots<i_{j_h}\leq i_k$ of  $i_1,\dots, i_k$,    one has
$$\sum_{\mathbf s\in S_k(s(w))} u^{(\mathbf s)}_{i_{j_1}}\cdots u^{(\mathbf s)}_{i_{j_{h}}}=|u_{i_{j_{h}}}|\cdots |u_{i_{j_{h}}}| \sum_{(s_{1},\dots,s_{k})\in S_k(s)} s_{{j_{h}}}\cdots s_{{j_{h}}}=0.$$
 Notice that $2^{k-1}$ is the cardinality of $S_k(s(w))$. Hence by the definition of near-control-affine system it easily follows that
\begin{align*}
\displaystyle  \sum_{\mathbf s\in S_k(s(w))}& \frac{1}{2^{k-1}}\F(x,u^{(\mathbf s)})= f_0(x)
+ 
\\\,\\ &\sum_{h=1}^{k} \frac{1}{2^{k-1}}\left( \sum_{   i_{1}\leq i_{j_1}<\cdots<i_{j_h}\leq i_k  }  \left( \sum_{\mathbf s\in S_k(s(w))} u^{(\mathbf s)}_{i_{j_1}}\cdots u^{(\mathbf s)}_{i_{j_{h}}} \right)f_{{\bf{e}}_{i_{j_1}}+\dots+{\bf{e}}_{i_{j_h}}}(x)\right)\\\,\\
\displaystyle =&f_0(x)+\frac{1}{2^{k-1}}\left( \sum_{   i_{1}<\cdots<  i_k  }  \left( \sum_{\mathbf s\in S_k(s(w))} u^{(\mathbf s)}_{i_{1}}\cdots u^{(\mathbf s)}_{i_{{k}}} \right)f_{{\bf{e}}_{i_{1}}+\dots+{\bf{e}}_{i_{k}}}(x)\right)			\\\,\\
\displaystyle =&f_0(x)+ w~f_{\alpha_1,\dots,\alpha_m}(x),
 \end{align*}
   which  concludes the proof of {\it Claim B}.
\vskip0.4truecm

To end  the proof of Lemma \ref{lnear-control-affine}  in case $K=(1,\dots,1)$,
it suffices to remark that for every $k=1,\dots,d$,  by the definition of $\bar r$ given in \eqref{rbaldef} 
$$[-\bar r, \bar r]\subseteq M [-r^k,r^k].$$
Therefore {\it Claim B}  implies that for every 
$$w=(w_{\mathbf e_1},\dots,w_{\mathbf e_d},w_{\mathbf e_1+\mathbf e_2},w_{\mathbf e_1+\mathbf e_3},\dots,w_{\mathbf e_1+\cdots+\mathbf e_d})\in [-\bar r,\bar r]^M=\bar U_r$$
\begin{align*}
{\F}_{\text{\it aff}}(x,w)&=
&\sum_{k=1}^d \sum_{i_1<\dots<i_k} \frac{1}{M}(f_0(x)+ M w_{\mathbf e_{i_1}+\dots+\mathbf e_{i_k}} f_{\mathbf e_{i_1}+\dots+\mathbf e_{i_k}}(x))
\in co~\F(x,U_r).
\end{align*}


\begin{thebibliography}{10}

\bibitem{Akh65}
{\sc Akhiezer, N.~I.}
\newblock {\em {The classical moment problem: and some related questions in
  analysis}}, vol.~5.
\newblock Oliver \& Boyd Edinburgh, 1965.

\bibitem{aldobressan}
{\sc Bressan, A.}
\newblock Hyperimpulsive motions and controllizable coordinates for lagrangian
  systems,.
\newblock {\em Atti Accad. Naz. dei Lincei, Memorie classe di Scienze Mat. Fis.
  Nat. Serie}, Memorie, Serie VIII, Vol XIX (1990), 197--246.

\bibitem{BR88}
{\sc Bressan, A., and Rampazzo, F.}
\newblock {On differential systems with vector-valued impulsive controls}.
\newblock {\em Boll. Un. Matematica Italiana 2}, B (1988), 641--656.

\bibitem{BR10}
{\sc Bressan, A., and Rampazzo, F.}
\newblock {Moving constraints as stabilizing controls in classical mechanics}.
\newblock {\em Archive for rational mechanics and analysis 196}, 1 (2010),
  97--141.

\bibitem{Dyk90}
{\sc Dykhta, V.~A.}
\newblock {Impulse-trajectory extension of degenerate optimal control
  problems}.
\newblock {\em IMACS Ann. Comput. Appl. Math 8\/} (1990), 103--109.

\bibitem{Ego02}
{\sc Egozcue, J., et~al.}
\newblock {From a nonlinear, nonconvex variational problem to a linear, convex
  formulation}.
\newblock {\em Applied Mathematics \& Optimization 47}, 1 (2002), 27--44.

\bibitem{Fed69}
{\sc Federer, H.}
\newblock {\em {Geometric measure theory}}.
\newblock Springer New York, 1969.

\bibitem{chez}
{\sc Goncharova, E., and Staritsyn, M.}
\newblock {Optimal control of dynamical systems with polynomial impulses}.
\newblock {\em Discrete and Continuous Dynamical Systems 35}, 9 (2015),
  4367--4384.

\bibitem{Gur72}
{\sc Gurman, V.~I.}
\newblock {Optimally Controlled Processes with Unbounded Derivatives}.
\newblock {\em Automation and Remote Control 33}, 12 (1972), 1924--1930.

\bibitem{Mez04}
{\sc Meziat, R.~J.}
\newblock {Analysis of non convex polynomial programs by the method of
  moments}.
\newblock In {\em {Frontiers in global optimization}}. Springer, 2004,
  pp.~353--371.

\bibitem{MR03}
{\sc Miller, B., and Rubinovich, E.~Y.}
\newblock {\em {Impulsive control in continuous and discrete-continuous
  systems}}.
\newblock Springer, 2003.

\bibitem{Mil96}
{\sc Miller, B.~M.}
\newblock {The generalized solutions of nonlinear optimization problems with
  impulse control}.
\newblock {\em SIAM journal on control and optimization 34}, 4 (1996),
  1420--1440.

\bibitem{mot04}
{\sc Motta, M.}
\newblock {Viscosity solutions of HJB equations with unbounded data and
  characteristic points}.
\newblock {\em Applied Mathematics and Optimization 49}, 1 (2004), 1--26.

\bibitem{MR13}
{\sc Motta, M., and Rampazzo, F.}
\newblock {Asymptotic controllability and optimal control}.
\newblock {\em Journal of Differential Equations 254}, 7 (2013), 2744--2763.

\bibitem{MS14}
{\sc Motta, M., and Sartori, C.}
\newblock {On asymptotic exit-time control problems lacking coercivity}.
\newblock {\em ESAIM: Control, Optimisation and Calculus of Variations 20}, 04
  (2014), 957--982.

\bibitem{Nat57}
{\sc Natanson, I.~P.}
\newblock {Theory of functions of a real variable}.
\newblock {\em Frederick Ungar, New York (1955/1961) 1}, 1 (1957), 2--4.

\bibitem{PT09}
{\sc Pedregal, P., and Tiago, J.}
\newblock {Existence results for optimal control problems with some special
  nonlinear dependence on state and control}.
\newblock {\em SIAM Journal on Control and Optimization 48}, 2 (2009),
  415--437.

\bibitem{RS00}
{\sc Rampazzo, F., and Sartori, C.}
\newblock {Hamilton-Jacobi-Bellman equations with fast gradient-dependence}.
\newblock {\em Indiana University Mathematics Journal 49}, 3 (2000),
  1043--1078.

\bibitem{Ris65}
{\sc Rishel, R.~W.}
\newblock {An extended Pontryagin principle for control systems whose control
  laws contain measures}.
\newblock {\em Journal of the Society for Industrial \& Applied Mathematics,
  Series A: Control 3}, 2 (1965), 191--205.

\bibitem{Sho50}
{\sc Shohat, J.~A., and Tamarkin, J.~D.}
\newblock {\em {The Problem of Moments}}.
\newblock No.~1. American Mathematical Soc., 1943.

\bibitem{VP88}
{\sc Vinter, R.~B., and Pereira, F.~L.}
\newblock {A maximum principle for optimal processes with discontinuous
  trajectories}.
\newblock {\em SIAM journal on control and optimization 26}, 1 (1988),
  205--229.

\end{thebibliography}
\end{document}